\documentclass[11pt,reqno,tbtags,draft]{amsart}
\usepackage{amssymb}
\usepackage{url}
\usepackage[square,numbers]{natbib}
\usepackage{bm}
\bibpunct[, ]{[}{]}{;}{n}{,}{,}

\title[Positive autocorrelation:\ random walk Metropolis--Hastings]
{Positive autocorrelation at unit lag for stationary random walk Metropolis--Hastings in $\bbR^d$
}

\newcommand\urladdrx[1]{{\urladdr{\def~{{\tiny$\sim$}}#1}}}
\author{James Allen Fill}
\address{Department of Applied Mathematics and Statistics,
The Johns Hopkins University,
3400 N.~Charles Street,
Baltimore, MD 21218-2682 USA}
\email{jimfill@jhu.edu}
\urladdrx{http://www.ams.jhu.edu/~fill/}
\thanks{Research for the first author supported by
the Acheson~J.~Duncan Fund for the Advancement of Research in
Statistics.}

\author{Svante Janson}
\thanks{Research for the second author supported by the Knut and Alice Wallenberg Foundation
and
the Swedish Research Council (VR)
}
\address{Department of Mathematics, Uppsala University, PO Box 480,
SE-751~06 Uppsala, Sweden}
\email{svante.janson@math.uu.se}
\urladdrx{http://www2.math.uu.se/~svantejs}

\makeatletter
\@namedef{subjclassname@2020}{\textup{2020} Mathematics Subject Classification}
\makeatother


\numberwithin{equation}{section}

\allowdisplaybreaks

\theoremstyle{plain}
\newtheorem{theorem}{Theorem}[section]
\newtheorem{lemma}[theorem]{Lemma}
\newtheorem{proposition}[theorem]{Proposition}
\newtheorem{corollary}[theorem]{Corollary}
\newtheorem{conj}[theorem]{Conjecture}

\theoremstyle{definition}
\newtheorem{example}[theorem]{Example}

\newtheorem{remark}[theorem]{Remark}

\theoremstyle{remark}

\newenvironment{romenumerate}[1][-10pt]{
\addtolength{\leftmargini}{#1}\begin{enumerate}
 }{\end{enumerate}}

\newenvironment{alphiienumerate}[1][-3pt]{
\addtolength{\leftmarginii}{#1}\begin{enumerate}

 }{\end{enumerate}}

\newcounter{oldenumi}
{\setcounter{oldenumi}{\value{enumi}}
\begin{romenumerate} \setcounter{enumi}{\value{oldenumi}}}
{\end{romenumerate}}

\newcounter{thmenumerate}

\newcounter{xenumerate}   

\newcommand{\refT}[1]{Theorem~\ref{#1}}
\newcommand{\refC}[1]{Corollary~\ref{#1}}
\newcommand{\refL}[1]{Lemma~\ref{#1}}
\newcommand{\refR}[1]{Remark~\ref{#1}}
\newcommand{\refS}[1]{Section~\ref{#1}}

\newcommand{\refP}[1]{Proposition~\ref{#1}}

\newcommand{\refE}[1]{Example~\ref{#1}}

\newcommand{\refApp}[1]{Appendix~\ref{#1}}
\newcommand{\refConj}[1]{Conjecture~\ref{#1}}

\begingroup
  \divide\count255 by 60
  \multiply\count255 by -60
  \advance\count255 by \time
  \ifnum \count255 < 10 \xdef\klockan{\the\count1.0\the\count255}
  \else\xdef\klockan{\the\count1.\the\count255}\fi
\endgroup

\newcommand\nopf{\qed}   

\newcommand{\sumnz}{\sum_{n\in\bbZ}}

\newcommand\set[1]{\ensuremath{\{#1\}}}

\newcommand\bigpar[1]{\bigl(#1\bigr)}

\newcommand\bigsqpar[1]{\bigl[#1\bigr]}
\newcommand\Bigsqpar[1]{\Bigl[#1\Bigr]}

\newcommand\xcpar[1]{\{#1\}}
\newcommand\bigcpar[1]{\bigl\{#1\bigr\}}

\def\rompar(#1){\textup(#1\textup)}    

\newcommand\innprod[1]{\langle#1\rangle}

\def\xexp(#1){e^{#1}}
\newcommand\ceil[1]{\lceil#1\rceil}
\newcommand\floor[1]{\lfloor#1\rfloor}

\newcommand\norm[1]{\|#1\|}

\newcommand\punkt{.\spacefactor=1000}    

\newcommand\ie{i.e\punkt}

\newcommand{\as}{a.s\punkt}

\newcommand\eqd{\overset{\mathrm{d}}{=}}

\newcommand\bbR{\mathbb R}

\newcommand\bbN{\mathbb N}

\newcommand\bbZ{\mathbb Z}

\newcounter{CC}
\newcounter{cc}

\newcommand\E{\operatorname{\mathbb E{}}}
\renewcommand\P{\operatorname{\mathbb P{}}}
\renewcommand\L{\operatorname{L}}
\newcommand\Var{\operatorname{Var}}
\newcommand\Cov{\operatorname{Cov}}
\newcommand\trace{\operatorname{tr}}
\newcommand\trCov{\operatorname{trCov}}
\newcommand\trVar{\operatorname{trVar}}
\newcommand\trCorr{\operatorname{trCorr}}
\newcommand\Corr{\operatorname{Corr}}

\newcommand\sgn{\operatorname{sgn}}

\newcommand\ga{\alpha}
\newcommand\gb{\beta}

\newcommand\gD{\Delta}

\newcommand\gam{\gamma}

\newcommand\gl{\lambda}

\newcommand\gs{\sigma}

\newcommand\gth{\theta}
\newcommand\eps{\varepsilon}

\newcommand\bX{\overline X}
\newcommand{\bXX}{\overline \XX}
\newcommand{\bpi}{\overline \pi}
\newcommand{\bphi}{\overline \phi}
\newcommand{\hrho}{\hat{\rho}}
\newcommand{\trho}{\tilde{\rho}}

\newcommand\cL{{\mathcal L}}

\newcommand\bb{\mathbf b}
\newcommand\cc{\mathbf c}   
\newcommand\ccc{\mathbf c}   
\newcommand\ee{\mathbf e}
\newcommand\ff{\mathbf f}

\newcommand\mm{\mathbf m}

\newcommand\xx{\mathbf x}
\newcommand\yy{\mathbf y}
\newcommand\zz{\mathbf z}

\newcommand\VV{\mathbf V}

\newcommand\XX{\mathbf X}
\newcommand\YY{\mathbf Y}
\newcommand\ZZ{\mathbf Z}

\newcommand\zzero{\mathbf 0}

\newcommand\indic[1]{\boldsymbol1\xcpar{#1}}

\newcommand\etta{\boldsymbol1}

\newcommand\qq{^{1/2}}

\newcommand\oi{[0,1]}

\newcommand\dtv{d_{\mathrm{TV}}}

\newcommand\dd{\,\mathrm{d}}

\newcommand\lhs{left-hand side}
\newcommand\rhs{right-hand side}

\newcommand\tK{{\widetilde K}}
\newcommand\tM{{\widetilde M}}

\newcommand\tX{{\widetilde X}}

\newcommand\tY{{\widetilde Y}}
\newcommand\tYY{{\widetilde \YY}}

\newcommand\hpi{\hat{\pi}}

\newcommand\hX{\widehat X}

\newcommand\hXX{\widehat \XX}

\newcommand\tb{\tilde b}
\newcommand\tg{\tilde g}
\newcommand\tth{\tilde h}

\newcommand\doi{D_{01}}

\newcommand\tr{\tilde r}

\newcommand{\ignore}[1]{}
\newcommand\tf{\tilde f}

\newcommand\pif{f}
\newcommand\tpif{\tf}

\usepackage{tikz}
\usepackage{amssymb}
\usetikzlibrary{arrows,positioning}
\usepackage[english]{babel}
\usepackage{pgfplotstable}
\usepackage{pgfplots}
\usepackage{adjustbox}
\pgfplotsset{compat=1.3}

\begin{document}

\date{January~27, 2026}

\begin{abstract}
It is often asserted in the literature that one should expect positive autocorrelation for random walk Metropolis--Hastings (RWMH), especially if the typical proposal step-size is small relative to the variability in the target density.  In this paper, we consider a stationary RWMH chain~$\XX$ taking values in 
$d$-dimensional Euclidean space and (subject only to the existence of densities with respect to Lebesgue measure) with general target distribution having finite second moment and general proposal random walk step-distribution.  We prove, for any nonzero vector~$\ccc$, strict positivity of the autocorrelation function at unit lag for the stochastic process $\langle \ccc, \XX \rangle$, that is,
\[
\Corr(\langle \ccc, \XX_0 \rangle, \langle \ccc, \XX_1 \rangle) > 0,
\]
and we establish the same result, but with weak inequality (which can in some cases be equality) when the state space for~$\XX$ is changed to the integer grid $\bbZ^d$.
Further, for $\ccc \neq \zzero$ we establish the sharp lower bound
\[
\Corr(\langle \ccc, \XX_0 \rangle, \langle \ccc, \XX_1 \rangle) > \tfrac19
\]
on autocorrelation when we assume both that (i)~the target density~$\pi$ is spherically symmetric and unimodal in the specific sense that $\pi(\xx) = \hpi(\|\xx\|)$ for some nonincreasing function $\hpi$ on $[0, \infty)$ and that (ii)~the proposal step-density is symmetric about~$\zzero$.

We study the autocorrelation indirectly, by considering the incremental
variance function (or incremental second-moment function) at unit lag.  The
same approach allows us also for $r \in [2, \infty)$ to upper-bound the
incremental $r$th-absolute-moment function at unit lag.  

We give also closely related inequalities for the total variation distance
between 
two distributions on $\bbR^d$ differing
only by a location shift.
\end{abstract}

\maketitle

\newpage

\subsection*{Outline of paper.}
In \refS{S:setup} we 
set up the random walk
Metropolis--Hastings chain~$\XX$ and discuss absolute moments in general
finite dimension~$d$.  In \refS{S:poscorr}, for~$\XX$ we prove strictly
positive unit-lag auto-trace-correlation as defined in \refS{S:setup}; the
main results of the paper (\refT{T:main} and \refC{C:maincov}), in 
a suitably absolutely continuous setting, are found in that section.
Relating
auto-trace-correlation to the more familiar autocorrelation for (nonzero) linear
combinations, we prove strict positivity at unit lag for the latter in
\refS{S:poslin}; see \refT{T:poslin} and \refC{C:poslin}.  
If we restrict to suitable unimodal target distributions and
symmetric proposal step-densities, we can sharpen the lower bound on autocorrelation; this is the subject of \refS{S:unimodaltarget}.
Results in a discrete setting analogous to those in Sections~\ref{S:poscorr}--\ref{S:unimodaltarget} are presented in \refS{S:I}.
Sections~\ref{S:setup}--\ref{S:I} all concern unit lag; in \refS{S:even} we prove an easy and much stronger result for even-order lags.  Appendices
\ref{A:TVLB}--\ref{A:TVLBZ} establish generalized moment lower bounds on
total variation distance between two distributions on $\bbR^d$ differing
only by a location shift; such bounds are central to the proof of the main
\refT{T:main} and its discrete analogue \refT{T:Imain}.  \refApp{A:more TV}
improves one of the main results of \refApp{A:TVLB}, \refC{C:TVLB}, in two
different ways.  \refApp{A:symm}  discusses the effect of symmetrization of either the target distribution or the proposal step-distribution on incremental $r$th absolute moment.  \refApp{A:IRWMHunimodalr} treats the case of symmetric discrete unimodal target and odd proposal steps.

\subsection*{Acknowledgment and closely related work.}
We are grateful to Jingyi Zhang and James~C.\ Spall for introducing us to consideration of positive autocorrelation at unit lag for stationary random walk Metropolis--Hastings, for providing us (at the same time) with a derivation of the $1$-dimensional instance of \refC{C:SMHcov}, and for pointing out (as we do in our \refE{E-}; see also 
\cite[penultimate paragraph of Section~1]{Zhang(2025)}) that the autocorrelation at unit lag can be strictly negative for a stationary Metropolis--Hastings chain that is not a random walk Metropolis--Hastings chain.  The main result, Theorem~3.5, of~\cite{Zhang(2025)}, obtained with the assistance of the first author of the present paper, gives an alternative expression to our~\eqref{covsymmuni} and a proof of positive correlation in the setting of our \refC{C:symmunimodal}.

Our paper does not discuss at all the question of optimal proposal design discussed in 
\cite[Sections 4--5]{Zhang(2025)}.

\section{The Metropolis--Hastings set-up and absolute moments} \label{S:setup}

\subsection*{Outline of \refS{S:setup}. }
\refS{SS:setting} 
sets up the stationary Metropolis--Hastings chain~$\XX$.  In \refS{SS:moments} we discuss moments for random vectors, introducing some apparently new terminology along the way.  In Sections \ref{SS:SMHcov}--\ref{SS:SRWMHsymmphicov} we write down expressions for our moments of interest, first under quite minimal assumptions about~$\XX$, then (as is our main interest) when the proposal chain has a random walk structure, and finally when the proposal chain has a symmetric random walk structure.

\subsection{The setting} \label{SS:setting}

Dongarra and Sullivan~\cite{DS(2000)} cite ``Metropolis Algorithm for Monte Carlo'' \cite{MRRTT(1953)} as one of the ten algorithms ``with the greatest influence on the development and practice of science and engineering in the 20th century.''  This algorithm, which plays a central role in Markov Chain Monte Carlo (MCMC) and (among other various uses) is widely used both in discrete settings in combinatorial probability and computer science and in continuous settings in Bayesian inference, certainly has not declined in importance in the 21st century; indeed, for example, a MathSciNet search for the word ``Metropolis'' appearing in any search field for the years 2000--2025 conducted on August~27,~2025 returned 3,978 hits.  The algorithm, generalized by Hastings~\cite{Hastings(1970)}, provides a method for converting a given ``proposal'' Markov chain into one with a specified ``target'' stationary distribution, and the so-converted chain in practice is used for (usually approximate) sampling from the target distribution.  (Working in a discrete setting, Metropolis et al.~\cite{MRRTT(1953)} assumed that the proposal transition matrix is symmetric; Hastings~\cite{Hastings(1970)} removed this restriction.) 
Briefly (see below for more details), the Metropolis--Hastings (MH) algorithm
is a Markov chain (the MH chain) where, given a state~$\xx$, we first sample
a state~$\yy$ from the given proposal distribution $Q(\xx,\cdot)$, then
compute an acceptance probability $\ga(\xx,\yy)$, and finally move to state~$\yy$ 
with this probability $\ga(\xx,\yy)$, otherwise remaining at~$\xx$.

We will set up a Metropolis--Hastings (MH) chain $\XX = (\XX_t)_{t \geq 0}$
in discrete time, assuming, unless otherwise specified, the existence of
certain densities, as we will explain; we will also (later, separately, see
\refR{R:discrete set-up} and \refS{S:I}) 
discuss a discrete analogue. 
We will specialize to \emph{random walk} MH (RWMH) a bit later, in \refS{SS:SRWMHcov}.  Our MH setting, to be spelled out in full here, is essentially that of \cite[Sections 2.3.1 and 2.3.2 (for MH and RWMH, respectively)]{Tierney(1994)}, an important special case of \cite{Tierney(1998)} (see case~1, ``Common dominating measure'', of Section~2 therein), taking the dominating measure (called~$\mu$ in~\cite{Tierney(1994)} and~$\nu$ in~\cite{Tierney(1998)}) to be Lebesgue measure on $\bbR^d$. 

We take the state space for~$\XX$ to be $\bbR^d$ (for some positive integer~$d$) with its usual inner (or dot) product $\langle \cdot, \cdot \rangle$ and norm $\| \cdot \|$, equipped with its Borel $\gs$-algebra.

Unless otherwise specified, 
{\bf we assume that 
the target distribution is absolutely continuous and has density~$\pi$} 
with respect to Lebesgue measure $\gl \equiv \gl_d$ on~$\bbR^d$.  
Similarly,
we assume that all the transition proposal distributions $Q(\xx, \cdot)$
[where $Q(\xx, B)$ is the transition proposal density from~$x$ to the Borel
set~$B$] have densities $q(\xx, \cdot)$ with respect to~$\gl$.  We assume
joint measurability of $q(\xx, \yy)$ in $(\xx, \yy)$.

We write
\begin{equation}
\label{ga}
\ga(\xx, \yy) = 
\begin{cases}
\min\left\{ 1, \frac{\pi(\yy) q(\yy, \xx)}{\pi(\xx) q(\xx, \yy)} \right\}, & \mbox{if $\pi(\xx) q(\xx, \yy) > 0$,} \\
1, & \mbox{if $\pi(\xx) q(\xx, \yy) = 0$} 
\end{cases}
\end{equation}
for the probability of accepting a proposed transition from~$\xx$ to~$\yy$.  If the proposed transition 
from~$\xx$ is rejected, the chain~$\XX$ remains 
at state~$\xx$.

The transition distribution $K(\xx, \cdot)$ for the MH chain from state~$\xx$ is thus a mixture of a density $k(\xx, \cdot)$ with respect to~$\gl$ and unit mass at~$\xx$, where
\begin{align}
k(\xx, \yy) 
&= q(\xx, \yy) \ga(\xx, \yy) \nonumber \\ 
&= 
\begin{cases}
\min\left\{ q(\xx, \yy), \frac{\pi(\yy) q(\yy, \xx)}{\pi(\xx)} \right\}, & \mbox{if $\pi(\xx) q(\xx, \yy) > 0$,} \\
q(\xx, \yy), & \mbox{if $\pi(\xx) q(\xx, \yy) = 0$.}
\end{cases}  
\label{k} 
\end{align}
It is a key observation for the MH algorithm that~$K$ is reversible with respect to~$\pi$ and thus has~$\pi$ as a stationary distribution; indeed, for $\xx \neq \yy$ we have
\begin{equation}
\pi(\xx) k(\xx, \yy)
= \min\left\{ \pi(\xx) q(\xx, \yy), \pi(\yy) q(\yy, \xx) \right\} \\
= \pi(\yy) k(\yy, \xx).  
\end{equation}
It is also important when using the MH algorithm for $\pi$-sampling purposes that the kernel~$K$ be appropriately irreducible (see, for example, discussion in \cite[Section~3]{Tierney(1994)}), but irreducibility conditions will play no role in this paper and we do not assume any such condition. 

By a \emph{stationary} MH chain, we will mean that $\XX_0$ (and hence also every $\XX_t$) has the density~$\pi$ with respect to~$\gl$.  {\bf We assume stationarity of the chain~$\XX$ henceforth.}  Thus the joint distribution of $\XX_0$ and $\XX_1$ puts some mass on the diagonal $\XX_0 = \XX_1$, and the remainder of the mass has joint density
\begin{equation}
\pi(\xx) k(\xx, \yy)
= \min\left\{ \pi(\xx) q(\xx, \yy), \pi(\yy) q(\yy, \xx) \right\}
\end{equation}
at pairs $(\xx, \yy)$ with $\xx \neq \yy$.  In particular, $\XX_0$ and $\XX_1$ form an exchangeable pair.

\begin{remark}
As noted by Peskun~\cite{Peskun(1973)} (see also the last paragraph of \cite[Section~3]{Tierney(1998)}), for any $\pi$-reversible Markov kernel based on acceptance of the same proposed $Q$-transitions from~$\xx$ to~$\yy$ now with probabilities $\gb(\xx, \yy)$ we have $\gb(\xx, \yy) \leq \ga(\xx, \yy)$ for all $(\xx, \yy)$ with $\xx \neq \yy$.

All the results of this paper will be based on upper bounds on expressions involving nonnegative multiples of the acceptance probabilities $\ga(\xx, \yy)$ with $\xx \neq \yy$ and so also hold when~$\ga$ is replaced 
by~$\gb$.
\end{remark}

\begin{remark}
\label{R:discrete set-up}
The absolutely continuous setting described in this subsection will be the focus of our main results, but we can also establish analogous results in the \emph{integer-grid Metropolis--Hastings} (IMH) setting,
where the state space for~$\XX$ is the integer grid $\bbZ^d$ (for some
positive integer~$d$), we write~$\pi$ for the target probability mass
function, $q$ is a (proposal) transition matrix on $\bbZ^d$, the acceptance
probability is again given by the expression~\eqref{ga}, and the IMH
transition matrix has off-diagonal elements given by~\eqref{k}.  See
\refS{S:I}
and \refApp{A:IRWMHunimodalr}.

This is another important special case of \cite{Tierney(1998)} (see case~1, ``Common dominating measure'', of Section~2 therein), now taking the dominating measure to be counting measure on $\bbZ^d$. 
\end{remark}

\subsection{Moments for random vectors} \label{SS:moments}

Recall for a real-valued random variable~$Y$ that the mean $\E Y \in \bbR$ is well defined when $\E |Y| < \infty$, and, for 
$r \in (0, \infty)$ that $Y \in \L^r$ means that $\E |Y|^r < \infty$.

For a random vector $\YY = (Y_1, \ldots, Y_d)$, the mean vector $\E \YY \in \bbR^d$ is defined as the vector 
$(\E Y_1, \ldots, \E Y_d)$ of means when each $Y_i \in L^1$; we will write $\YY \in L^1$ for this condition, which is equivalent to $\E \| \YY \| < \infty$.  If $\YY \in L^1$, then for any 
measurable mapping~$W$ we have the following ``law of total expectation'':
\begin{equation}
\label{LoTE}
\E \YY = \E [\E(\YY \mid W)].
\end{equation} 

Similarly, given $r \in (0, \infty)$ one 
can show that $\YY \in L^r$, which we define as meaning that each $Y_i \in L^r$, if and only if $\E \| \YY \|^r$, which we will call the \emph{$r$th absolute moment of~$\YY$}, is finite.  If $\VV = (\VV_t)_{t \geq 0}$ is (like a stationary MH chain) a (strict-sense) stationary stochastic process with $\VV_0 \in L^r$, we will call 
$\E \| \VV_{s + t} - \VV_s \|^r \equiv \E \| \VV_t - \VV_0 \|^r$ its \emph{incremental $r$th-absolute-moment function} at \emph{lag}~$t$. 

If~$\YY$ and~$\tYY$ are two $d$-dimensional random vectors in $L^2$, their cross-covariance matrix $\Cov(\YY, \tYY)$ is the $d \times d$ matrix with $(i, j)$-entry equal to $\Cov(Y_i, \tY_j)$.  We then define the \emph{trace-covariance} 
$\trCov(\YY, \tYY)$ of~$\YY$ and~$\tYY$ to be the trace of their cross-covariance matrix, namely, the real number
\begin{align}
\trCov(\YY, \tYY) 
&= \trace \Cov(\YY, \tYY) 
= \sum_{i = 1}^d \Cov(Y_i, \tY_i) \notag\\
&= \E \langle \YY - \E \YY,\,\tYY - \E \tYY \rangle
= \E \langle \YY, \tYY \rangle - \langle \E \YY, \E \tYY \rangle
\end{align}
and the \emph{trace-variance} $\trVar \YY \in [0, \infty)$ of~$\YY$ to be the trace of the covariance matrix of $\YY \in L^2$:
\begin{align}\label{trVar}
  \trVar \YY = \trCov(\YY, \YY) = \sum_{i = 1}^d \Var Y_i = \E \| \YY - \E \YY \|^2 = \E \| \YY \|^2 - \|\!\E \YY \|^2. 
\end{align}
There is a simple relationship among trace-covariance, trace-variance, and second absolute moment, which is simplest when [as is true when $(\YY, \tYY) = (\XX_0, \XX_1)$] the two random vectors~$\YY$ and~$\tYY$ involved are assumed to have the same distribution, resulting in that case in
\begin{align}
\trCov(\YY, \tYY) 
&= \sum_{i = 1}^d \Cov(Y_i, \tY_i)
= \frac12 \sum_{i = 1}^d [2 \Var Y_i - \E (Y_i - \tY_i)^2] \nonumber \\
&= \frac12 (2 \trVar \YY - \E \| \YY - \tYY \|^2). \label{Pythagoras}  
\end{align}
Further, for any measurable mapping~$W$ we have the following ``law of total trace-covariance'':
\begin{equation}
\label{LoTC}
\trCov(\YY, \tYY) = \E \trCov(\YY, \tYY\,|\,W) + \trCov(\E(\YY\,|\,W), \E(\tYY\,|\,W)).
\end{equation}

Observe that if $\YY, \tYY \in L^2$, then
\begin{align}
| \trCov(\YY, \tYY) | 
&= |\!\E \langle \YY - \E \YY,\,\tYY - \E \tYY \rangle |
\leq \E | \langle \YY - \E \YY,\,\tYY - \E \tYY \rangle | \notag\\
&\leq \E (\| \YY - \E \YY \| \times \| \tYY - \E \tYY \|) \notag\\
&\leq (\E \| \YY - \E \YY \|^2 \times \E \| \tYY - \E \tYY \|^2)^{1/2} \notag\\
&= (\trVar \YY)^{1/2} (\trVar \tYY)^{1/2}. 
\end{align}
We are therefore led to define the \emph{trace-correlation} $\trCorr(\YY, \tYY)$ of $\YY, \tYY \in L^2$, when neither has a degenerate distribution, as
\begin{equation}
\label{corrdef}
\trCorr(\YY, \tYY) := \frac{\trCov(\YY, \tYY)}{(\trVar \YY)^{1/2} (\trVar \tYY)^{1/2}} \in [-1, 1].
\end{equation}
We caution that $\trCorr(\YY, \tYY)$ is \emph{not} in general the trace of the cross-correlation matrix of~$\YY$ and~$\tYY$.
Note that if~$\YY$ and~$\tYY$ are independent, then $\trCorr(\YY, \tYY) = 0$.  At the other extreme, 
$\trCorr(\YY, \tYY) = \pm 1$ if and only if $\tYY = a \YY + \bb$ almost surely for some $a \in \bbR$ with $\sgn a = \pm1$ and some $\bb \in \bbR^d$.

If $\VV = (\VV_t)_{t \geq 0}$ is (like a stationary MH chain) a (weak-sense) stationary stochastic process (with 
$\VV_0 \in L^2$ and $\trVar \VV_0 > 0$), we will call 
\begin{equation}
\trCov(\VV_s, \VV_{s + t}) \equiv \trCov(\VV_0, \VV_t)
\end{equation} 
its \emph{auto-trace-covariance function}, or, more simply, its \emph{autocovariance function}, and we will call 
\begin{equation}
\trCorr(\VV_s, \VV_{s + t}) \equiv \trCorr(\VV_0, \VV_t)
\end{equation}
its \emph{auto-trace-correlation function}, or, more simply, its \emph{autocorrelation function}, at \emph{lag}~$t$.

We will often assume that~$\pi$ has finite trace-variance $\gs^2$.  That is:\ {\bf We assume often that $\XX_0 \in L^2$ with $\gs^2 := \trVar \XX_0 \in (0, \infty)$}, the strict positivity of $\gs^2$ following from the existence of the 
density~$\pi$.  
At times we will make a more general assumption, namely, that for some particular value of $r \in (0, \infty)$ we have $\XX_0 \in L^r$.

A special case (\refC{C:maincov}) of our main theorem will show
that the autocorrelation at unit lag is always strictly positive for  a (stationary) RWMH chain with $\XX_0 \in L^2$.

\begin{example}\label{E-}
The
autocorrelation at unit lag can be strictly negative for an MH chain (with $\XX_0 \in L^2$) that is
not an RWMH chain.  
Consider first the following extreme discrete (IMH) counterexample (recall
\refR{R:discrete set-up}), a simple instance of case~2,
``Deterministic proposals'', of \cite[Section~2]{Tierney(1998)}.
Let $d = 1$ (although the counterexample is easily extended to any~$d$) and
consider the (now discrete) densities
\begin{align}
\pi(x)&
:=
\begin{cases}
\frac12, & \mbox{if $x \in \set{0, 1}$,} \\
0, & \mbox{otherwise;}
\end{cases}
\\
q(x, y)&
:=
\begin{cases}
1, & \mbox{if $y = 1 - x$,} \\
0, & \mbox{otherwise.}
\end{cases}
\end{align}
Then $\P(X_0 = 0, X_1 = 1) = \P(X_0 = 1, X_1 = 0) = 1/2$, so $\Corr(X_0, X_1) = -1$.
We may also obtain a counterexample with absolutely continuous target and
proposal distributions, as assumed  above, by perturbing this discrete example.
For example, let $X_0$ be uniform on $(-\eps,\eps)\cup(1-\eps,1+\eps)$,
and let $Q(x,\cdot)$ be uniform on $(1-\eta,1+\eta)$ for $x<\frac12$, and 
uniform on $(-\eta,\eta)$ for $x>\frac12$, with small $\eps,\eta>0$. 
Then, letting first $\eta\to0$ and then $\eps\to0$, we have
$\Corr(X_0, X_1) \to -1$.
\end{example}

(Strictly) positive autocorrelation at unit lag is something of an indicator of positive dependence between 
$\XX_0$ and $\XX_1$.  Another such indicator, for any given $r \in (0, \infty)$, is that
\begin{equation}
\label{ratio}
\frac{\E \| \XX_1 - \XX_0 \|^r}{\E \| \XX_0^* - \XX_0 \|^r} < 1,
\end{equation}
where $\XX_0^*$ is an independent copy of $\XX_0$.  Note that when $r = 2$, the inequality~\eqref{ratio} is equivalent to positive autocorrelation by~\eqref{Pythagoras}.

\subsection{An incremental $r$th-absolute-moment formula for general MH chains} \label{SS:SMHcov}

The following formula is immediate from \eqref{k}.

\begin{proposition}
\label{P:SMHr}
In the general MH
setting of \refS{SS:setting}, for $r \in (0, \infty)$ we 
have
\begin{align}
\lefteqn{\hspace{-.15in}\E \| \XX_1 - \XX_0 \|^r} \nonumber \\ 
&= \int\!\int\!\|\xx - \yy\|^r \min\{\pi(\xx) q(\xx, \yy), \pi(\yy) q(\yy, \xx)\} \dd \yy \dd \xx. \label{SMHr}
\end{align}
\nopf
\end{proposition}

The following covariance formula is immediate from \refP{P:SMHr} (with $r = 2$) and~\eqref{Pythagoras}.

\begin{corollary}
\label{C:SMHcov}
In the general MH
setting of \refS{SS:setting}, if $\XX_0 \in L^2$ then
\begin{align}
\lefteqn{\hspace{-.15in}\trCov(\XX_0, \XX_1)} \nonumber \\ 
&= \frac12 \left[ 2 \gs^2 - \int\!\int\!\|\xx - \yy\|^2 \min\{\pi(\xx) q(\xx, \yy), \pi(\yy) q(\yy, \xx)\} \dd \yy \dd \xx \right]. \label{SMHcov}
\end{align}
\nopf
\end{corollary}

\subsection{An incremental $r$th-absolute-moment formula for general
\mbox{RWMH} 
chains}\label{SS:SRWMHcov}

If the transition proposal densities $q(\xx, \yy)$ are of the \emph{random walk form} $\phi(\yy - \xx)$ for some 
density~$\phi$, 
then the chain is 
called \emph{random walk Metropolis--Hastings} (RWMH).

\begin{proposition}
\label{P:SRWMHr}
For RWMH and $r \in (0, \infty)$ we have
\begin{equation}
\label{SRWMHr}
\E \| \XX_1 - \XX_0 \|^r
= \int\!\|\zz\|^r \int\!\min\{\pi(\xx) \phi(-\zz), \pi(\xx - \zz) \phi(\zz)\} \dd \xx \dd \zz.
\end{equation} 
\end{proposition}

\begin{proof}
Recall~\eqref{SMHr} and use the translation invariance of Lebesgue measure (that is, make a change of variables).
\end{proof}

\begin{corollary}
\label{C:SRWMHcov}
For RWMH, if $\XX_0 \in L^2$ then
\begin{align}
\lefteqn{\hspace{-.05in}\trCov(\XX_0, \XX_1)} \nonumber \\ 
&= \frac12 \left[ 2 \gs^2 - \int\!\|\zz\|^2 \int\!\min\{\pi(\xx) \phi(-\zz), \pi(\xx - \zz) \phi(\zz)\} \dd \xx \dd \zz \right].
\label{SRWMHcov}
\end{align}
\nopf
\end{corollary}

\subsection{An incremental $r$th-absolute-moment formula for general 
\mbox{RWMH} 
chains with symmetric proposal step-density}\label{SS:SRWMHsymmphicov}

We say that the density $\phi$ is \emph{symmetric}
if $\phi(-\xx)=\phi(\xx)$, $\xx\in\bbR^d$.

Recall that if
$\YY$ and $\tYY$ are two $d$-dimensional random vectors having densities $f$
and $\tf$,
then the total variation distance between them 
(or rather between their distributions)  
can be computed as
\begin{align}\label{dtv}
  \dtv(\YY,\tYY) 
= \frac12\int|f(\xx)-\tf(\xx)|\dd\xx
=1-\int \min\bigcpar{f(\xx),\tf(\xx)}\dd\xx.
\end{align}
For any constant vector~$\zz$, $\pi(\cdot - \zz)$ is the density of $\XX_0+\zz$.
Hence, if $\phi$ is symmetric, then the inner integral
\begin{align}\label{dtv2}
 \int\!\min\{\pi(\xx) \phi(-\zz), \pi(\xx - \zz) \phi(\zz)\} \dd \xx
\end{align} 
in~\eqref{SRWMHr} and in~\eqref{SRWMHcov} 
equals $\phi(\zz) [1 - \dtv(\XX_0, \XX_0 + \zz)]$.
Consequently, in the case of symmetric~$\phi$, \refP{P:SRWMHr} and \refC{C:SRWMHcov}
can be reformulated as follows.

\begin{proposition}
\label{P:SRWMHsymmphir}
For RWMH with symmetric~$\phi$ and $r \in (0, \infty)$ we have
\begin{equation}
\label{SRWMHsymmphir}
\E \| \XX_1 - \XX_0 \|^r
= \int\!\|\zz\|^r \phi(\zz) [1 - \dtv(\XX_0, \XX_0 + \zz)] \dd \zz.
\end{equation} 
\nopf
\end{proposition}

\begin{corollary}
\label{C:SRWMHsymmphicov}
For RWMH 
with symmetric~$\phi$, if $\XX_0 \in L^2$ then
\begin{align}
\lefteqn{\hspace{-.05in}\trCov(\XX_0, \XX_1)} \nonumber \\ 
&= \frac12 \left[ 2 \gs^2 - \int\!\|\zz\|^2 \phi(\zz) [1 - \dtv(\XX_0, \XX_0 + \zz)] \dd \zz \right].
\label{SRWMHsymmphicov}
\end{align}
\nopf
\end{corollary}

\section{Positive unit-lag autocorrelation} \label{S:poscorr}
We return to the general absolutely continuous RWMH
setting of \refS{SS:SRWMHcov}.  
The following is one of our two main results, the other being its consequence \refC{C:maincov}.

\begin{theorem}
\label{T:main}
Let $r \in [2, \infty)$.  For 
RWMH,
if $\XX_0 \in L^r$ then for any $\mm \in \bbR^d$ we have
\begin{equation}
\label{rUB}
\E \| \XX_1 - \XX_0 \|^r < 2^{r - 1} \E \| \XX_0 - \mm \|^r.
\end{equation}
\end{theorem}

Of course, we may choose $\mm$ in~\eqref{rUB} to minimize 
$\E \| \XX_0 - \mm \|^r$.
In particular, for $r=2$, we may choose $\mm=\E\XX_0$ and then the bound in
\eqref{rUB} is $2^{r-1}\trVar\XX_0$ by \eqref{trVar}.

\begin{proof}
Define
\begin{equation}
\label{bphi}
\bphi(\zz) := \tfrac12 [\phi(-\zz) + \phi(\zz)].
\end{equation}
According to \refT{T:TVLB} (and homogeneity), the $\zz$-integrand
\begin{equation}
\|\zz\|^r \int\!\min\{\pi(\xx) \phi(-\zz), \pi(\xx - \zz) \phi(\zz)\} \dd \xx
\end{equation} 
appearing in~\eqref{SRWMHr} [and in~\eqref{SRWMHcov} for $r = 2$] is strictly smaller than 
\begin{align}
2^{r - 2} \E \| \XX_0-\mm \|^r [\phi(-\zz) + \phi(\zz)]
= 2^{r - 1} \E \| \XX_0-\mm \|^r \bphi(\zz),  
\end{align}
for every $\zz$ such that $\bphi(\zz)>0$.
The result follows from this by integrating, using \eqref{SRWMHr} and
$\displaystyle{\int\bphi(\zz)\dd\zz=1}$.
\end{proof}

\begin{remark}
\label{R:other ineqs}
From
\refT{T:main} we can of course obtain a host of additional inequalities by averaging over~$r$, provided only that $\mm \equiv \mm(r)$ is measurable in~$r$.  In particular, suppose for simplicity that we choose~$\mm$ independent of~$r$.  Let~$\mu$ be any probability measure concentrated on $[2, \infty)$ and~$\psi$ is its moment generating function.  Then, with the natural conventions $\ln 0 = - \infty$ and $\psi(- \infty) = 0$, from~\eqref{rUB} we obtain
\begin{equation}
\label{psi ineq}
\E \psi(\ln \| \XX_1 - \XX_0 \|) \leq \tfrac12 \E \psi(\ln (2 \| \XX_0 - \mm \|)),
\end{equation}
and the inequality is strict provided the \rhs{} is finite.  As one example, if~$\mu$ is the law of $W + 2$, where $W \sim \mbox{Poisson}(\gl)$ with $\gl \geq 0$, then~\eqref{psi ineq} yields equivalently
\begin{equation}
\E \big[ \| \XX_1 - \XX_0 \|^2 \exp(\gl \| \XX_1 - \XX_0 \|) \big] 
\leq 2 \E \big[ \| \XX_0 - \mm \|^2 \exp(2 \gl \| \XX_0 - \mm \|) \big]. 
\end{equation}    
\end{remark}

\begin{remark}
\label{R:Lipschitz}
Trivially, \refT{T:main} can be used also 
to bound $\E \| \ff(\XX_1) - \ff(\XX_0) \|^r$ for any Lipschitz function~$\ff$ from $\bbR^d$ to any normed space, since 
$\E \| \ff(\XX_1) - \ff(\XX_0) \|^r \leq c_{\ff}^r \E \| \XX_1 - \XX_0 \|^r$, where $c_{\ff} \in [0, \infty)$ is the Lipschitz constant for~$\ff$.
\end{remark}

\begin{remark}
\label{R=}
Although the inequality \eqref{rUB} is strict for every instance, the
constant $2^{r-1}$ is best possible.
Note first that
there are examples with
discrete $\XX_0$ and $\ZZ$ such that equality holds (for every $r$).
For example, let $d=1$ and let 
$\P(X_0=\pm1)=\frac12$, 
$\P(Z=\pm2)=\frac12$; then  $|X_1 - X_0|$ equals~$0$ or~$2$ with probability $\frac12$ each 
so equality holds in \eqref{rUB} with $m=0$.
(See also \refE{E:discrete}.)
We may perturb
this example as in \refE{E-} to obtain absolutely continuous examples
showing that \eqref{rUB} does not hold with a constant smaller than $2^{r-1}$.
\end{remark}

For $r=2$, the \rhs{} of \eqref{rUB} is minimized by $\mm=\E \XX_0$, and by
\eqref{Pythagoras} and 
\eqref{corrdef} [or by \eqref{SRWMHcov}] we obtain

\begin{corollary}
\label{C:maincov}
For RWMH, if $\XX_0 \in L^2$ then
\begin{equation}
\label{rLB}
\trCorr(\XX_0, \XX_1) > 0.
\end{equation}
\nopf
\end{corollary}

\begin{remark}
If $\XX_0 \in L^2$, an alternative lower bound on $\trCorr(\XX_0, \XX_1)$ can be obtained very simply from~\eqref{SRWMHcov} by upper-bounding the minimum expression by $\pi(\xx) \phi(-\zz)$.  We then find
\begin{equation}
\label{simpleLB}
\trCorr(\XX_0, \XX_1) \geq 1 - \frac{\E \| \ZZ \|^2}{\gs^2}.
\end{equation}
This lower bound can of course be negative (and can even be smaller than the smallest possible value $-1$ of a trace-correlation), but it is positive if and only if 
\begin{equation}
\label{2bounds}
\E \| \ZZ \|^2 < \gs^2.
\end{equation}
For example, if~$\ZZ$ has a symmetric distribution, then the lower bound~\eqref{simpleLB} is positive if and only if
\begin{equation}
\label{2vars}
\Var \ZZ < \Var \XX_0.
\end{equation}

A similar remark applies more generally to comparison of~\eqref{SRWMHr} and~\eqref{rUB}.
\end{remark}

The condition $r\ge2$ in \refT{T:main} is necessary in general, as shown in
\refR{R:TVLB-}. However,
in the special case of symmetric~$\phi$, we may extend the range of $r$ in \refT{T:main}.

\begin{theorem}
\label{T:main1}
For RWMH with symmetric $\phi$,
we have \eqref{rUB} for every $r \in [1, \infty)$ such that $\XX_0 \in L^r$.
\end{theorem}

\begin{proof}
We use \eqref{SRWMHsymmphir} and note that
$\|\zz\|^r [1 - \dtv(\XX_0, \XX_0 + \zz)]$
is by \refT{T:TVLB99} strictly smaller than 
$2^{r - 1} \E \| \XX_0 - \mm \|^r$ for almost every $\zz$.
The result follows 
by integrating.
\end{proof}

\begin{remark}
\refR{R:other ineqs} again applies, with~$\mu$ now allowed to be concentrated on $[1, \infty)$.
As one example, if~$\mu$ is the law of $W + 1$, where $W \sim \mbox{Poisson}(\gl)$ with $\gl \geq 0$, then we find
\begin{equation}
\E \big[ \| \XX_1 - \XX_0 \| \exp(\gl \| \XX_1 - \XX_0 \|) \big] 
\leq \E \big[ \| \XX_0 - \mm \| \exp(2 \gl \| \XX_0 - \mm \|) \big], 
\end{equation}
with strict inequality if the \rhs{} is finite.
\end{remark}

\begin{remark}
The example in \refR{R=} shows that the constant $2^{r-1}$ is best possible also in \refT{T:main1}. 
\end{remark}

\begin{remark}
Even in the case of symmetric $\phi$ as in \refT{T:main1},
the condition $r\ge1$ is necessary in general.
For $r<1$, a discrete counterexample is given by
$X_0$ uniform on $\set{0,1}$, $Z$ uniform on $\pm1$, and $m=0$,
see also \refR{R:TVLB2}. A counterexample with absolutely continuous $X_0$
and $Z$ (as assumed in the present section)
is as usual obtained by perturbing the discrete example.
\end{remark}

\section{Strictly positive unit-lag autocorrelation for nonzero linear combinations} \label{S:poslin}

\refT{T:main} 
and \refC{C:maincov} are somewhat unsatisfying.  For example, the conclusion of \refC{C:maincov} is that the unweighted \emph{average} of $\Cov(X_{0i}, X_{1i})$ over $i \in [d]$ is positive.  However, in this section we show, in the same general RWMH
setting as in~\refS{S:poscorr}, that in fact the results of that section extend, \emph{mutatis mutandis}, to individual coordinates and, more generally, to all linear combinations of the coordinates.

\begin{theorem}
\label{T:poslin}
Let $r \in [2, \infty)$, let~$m$ be any real number, and let~$\ccc$ be any
nonzero
vector in $\bbR^d$.  For
RWMH,
if\/ $\XX_0 \in L^r$ then
\begin{equation}
\label{rUBposlin}
\E | \langle \ccc, \XX_1 - \XX_0 \rangle |^r < 2^{r - 1} \E | \langle \ccc, \XX_0 \rangle - m |^r.
\end{equation}
\end{theorem}

\begin{proof}
For RWMH we have in analogy with \eqref{SRWMHr},
see also \eqref{SMHr},
\begin{equation}\label{ZRWMHr}
\E |\innprod{\ccc, \XX_1 - \XX_0} |^r
= \int\!|\innprod{\ccc,\zz}|^r \int\!\min\{\pi(\xx) \phi(-\zz), \pi(\xx -
\zz) \phi(\zz)\} \dd \xx \dd \zz.
\end{equation} 
Hence, \refT{T:TVLB} yields
\begin{align}\label{ZRWMHr2}
\E |\innprod{\ccc, \XX_1 - \XX_0} |^r&
< \int\!2^{r-2}\E|\innprod{c,\XX_0}-m|^r(\phi(\zz)+\phi(-\zz)) \dd\zz
\notag\\&
=2^{r-1}\E|\innprod{c,\XX_0}-m|^r.
\end{align}  
\end{proof}

\begin{corollary}
\label{C:poslin}
Let~$\ccc$ be any  nonzero vector in $\bbR^d$.  For 
\mbox{RWMH},
if $\XX_0 \in L^2$, then
\begin{equation}
\label{LBposlin}
\Cov(\langle \ccc, \XX_0 \rangle, \langle \ccc, \XX_1 \rangle) > 0.
\end{equation}
\end{corollary}

\begin{proof}
We  assume without loss of generality that $\E \XX_0 = \zzero$.  Use~\eqref{Pythagoras} with $d = 1$, 
$Y = \langle \ccc, \XX_0 \rangle$, and $\tY = \langle \ccc, \XX_1 \rangle$ and apply~\refT{T:poslin} with $m = 0$ to conclude~\eqref{LBposlin}.
\end{proof}

\section{The case of unimodal target distribution} \label{S:unimodaltarget}

\subsection*{Outline of \refS{S:unimodaltarget}. }
In \refS{SS:symmetrization}, in
the fairly general setting of~\refS{SS:SRWMHsymmphicov}, we show that
symmetrization of the target density only increases $r$th absolute
moment and decreases correlation.  In~\refS{SS:rsymmetricunimodal} we
specialize our moment formulas to the case that $d = 1$, the target
density~$\pi$ is symmetric and unimodal with mean~$0$, and the proposal
step-density $\phi$ is symmetric. 
In \refS{SS:1/9} we
prove in the setting just described the sharp lower bound  $\Corr(X_0, X_1)
> \frac19$ and some other similar but somewhat more general results, and
in \refS{SS:high} we generalize \refS{SS:1/9} to higher dimensions. 

\subsection{Target symmetrization} \label{SS:symmetrization}

The point of this subsection is that, in the setting of
\refS{SS:SRWMHsymmphicov} (symmetric proposal step-density), symmetrization
of the target density only increases incremental $r$th absolute moment and
decreases correlation.  Specifically, given a target density~$\pi$,
define another symmetric target density~$\bpi$ by
\begin{equation}
\label{bpi}
\bpi(\xx) := \tfrac12 [\pi(-\xx) + \pi(\xx)].
\end{equation}
Let $\bXX=(\bXX_t)$ denote the corresponding MH chain.

Note  first that symmetrization of~$\pi$ does not affect $\E \|\XX_0\|^r$.  For the effect on incremental $r$th absolute moment we have the following proposition.
  
\begin{proposition}
\label{P:r bar versus not}
Let $r\in(0,\infty)$.
For RWMH with symmetric $\phi$, we have
\begin{equation}
\label{r bar versus not}
\E \| \bXX_1 - \bXX_0 \|^r \geq \E \| \XX_1 - \XX_0 \|^r.  
\end{equation}
\end{proposition}
\begin{proof}
Because~$\phi$ is assumed symmetric, from~\eqref{SRWMHr} we have
\begin{equation}\label{sw1}
\E \| \XX_1 - \XX_0 \|^r
= \int\!\|\zz\|^r \phi(\zz) \int\!\min\{\pi(\xx), \pi(\xx - \zz)\} \dd \xx \dd \zz
\end{equation}
and
\begin{equation}\label{sw2}
\E \| \bXX_1 - \bXX_0 \|^r
= \int\!\|\zz\|^r \phi(\zz) \int\!\min\{\bpi(\xx), \bpi(\xx - \zz)\} \dd \xx \dd \zz.
\end{equation}
It is easy to see that
\begin{multline}\label{sw3}
\min\{\bpi(\xx), \bpi(\xx - \zz)\}
\\
\ge
\tfrac12\bigsqpar{\min\{\pi(\xx), \pi(\xx - \zz)\} + \min\{\pi(- \xx), \pi(- \xx + \zz)\}}.
\end{multline}
Substituting \eqref{sw3} into \eqref{sw2} 
shows that 
\begin{align}
\lefteqn{\E \| \bXX_1 - \bXX_0 \|^r} \nonumber \\
&\geq \int\!\|\zz\|^r \phi(\zz) \int\!\frac12 \left[ \min\{\pi(\xx), \pi(\xx - \zz)\} + \min\{\pi(- \xx), \pi(- \xx + \zz)\} \right]
\dd \xx \dd \zz.
\end{align}
This can be written as the average of two (double) integrals;
one is identical to \eqref{sw1}, and the other is equal to it by a change of
variables and the symmetry of $\phi$.
Hence, \eqref{r bar versus not} follows.
\end{proof}  

For a general discussion of the effect of symmetrization of~$\pi$ and/or $\phi$, see \refApp{A:symm}.

\subsection{An incremental $r$th-absolute-moment formula for RWMH
with symmetric unimodal target density and symmetric proposal density} \label{SS:rsymmetricunimodal}

A fruitfully different approach from that of \refS{S:poscorr} can be taken when the target density is symmetric and unimodal.
{\bf In this subsection we assume that $d = 1$, 
that the target density~$\pi$ is symmetric and unimodal with mean~$0$, and that the proposal step-density~$\phi$ is  symmetric (about~$0$).}

\begin{theorem}
\label{T:rsymmuni}
Let $r \in (0, \infty)$.  Consider 
RWMH with a target density~$\pi$ as described above and proposal step-density~$\phi$  symmetric (about~$0$).
Then
\begin{align}
\label{rsymmuni}
\E |X_1 - X_0 |^r &
= 2 \int_0^{\infty}\!\phi(z) z^r \P(|X_0| > z / 2) \dd z
=\E\bigsqpar{|Z|^r\P(|X_0| > |Z| / 2)},
\end{align}
where~$X_0$ and~$Z$ are independent and~$Z$ has density~$\phi$.
\end{theorem} 

\begin{proof}
We note from the symmetric unimodality of~$\pi$ that
\begin{equation}
\min\{\pi(x), \pi(y)\} = \pi(\max\{|x|, |y|\})
\end{equation}
and thus for $z \in (0, \infty)$ we have
\begin{equation}
1 - \dtv(X_0, X_0 + z) 
= \int\!\pi(\max\{|x|, |x - z|\}) \dd x
= \P(|X_0| > z / 2).
\end{equation}
Therefore~\eqref{rsymmuni} follows routinely from~\eqref{SRWMHsymmphir}.
\end{proof}

\begin{corollary}
\label{C:covsymmuni}
Consider RWMH
with a target density~$\pi$ as described above and proposal step-density~$\phi$  symmetric (about~$0$).  
If $X_0 \in L^2$, then
\begin{equation}
\label{covsymmuni}
\Cov(X_0, X_1) = \gs^2 - \int_0^{\infty}\!\phi(z) z^2 \P(|X_0| > z / 2) \dd z.
\end{equation}
\nopf
\end{corollary} 

\subsection{Correlation greater than $1/9$} \label{SS:1/9}

With the preparation of the preceding two subsections completed, we are nearly ready to state our most general theorem (\refT{T:unimodal}) for unimodal target densities.  {\bf In this subsection we assume that $d = 1$, that the target density~$\pi$ is unimodal with mode equal to some point $m \in \bbR$, and that the proposal step-density is symmetric  (about~$0$).}

\refT{T:unimodal} 
will follow rather easily from the following inequality of Winkler~\cite{Winkler(1866)} (sometimes also called the 
Camp--Meidell inequality \cite{Camp(1922), Meidell(1922)}), which extends
Gauss's inequality (the case $r = 2$).  
(Actually, Winkler \cite{Winkler(1866)} 
seems to consider only integer exponents $r$,
but the proof applies to all $r>0$; similarly
\cite{Camp(1922)} and \cite{Meidell(1922)} consider only even integers.
The general case is stated and proved in
\cite{DJ-D(1985)}.)
We include the proof, which is essentially the same as in \cite[Lemma~3]{DJ-D(1985)}, in order to emphasize when the inequality is an equality, a matter discussed for Gauss's inequality in \cite[Theorem~4.2]{Ion(2023)}.

\begin{lemma}[Winkler~\cite{Winkler(1866)}]
\label{L:Winkler}
Let $r \in (0, \infty)$.  If~$Y$ has a unimodal distribution about~$0$ and $y > 0$, then
\begin{equation}
\label{Winkler}
y^r \P(|Y| > y) \leq \left( \frac{r}{r + 1} \right)^r \E |Y|^r,
\end{equation}
with equality if and only if the distribution of~$Y$ is a mixture of unit mass at~$0$, 
{\rm Unif}$\left( - \frac{r + 1}{r} y, 0 \right)$, and {\rm Unif}$\left( 0, \frac{r + 1}{r} y \right)$.
\end{lemma}

\begin{proof}
According to the well-known \cite[Theorem~1.3]{DJ-D(1988)}, $Y$ has a unimodal distribution about~$0$ if and only if there exist independent random variables $U \sim \mbox{Unif}(0, 1)$ and~$W$ such that~$Y$ has the same distribution as $U W$.  We then compute each side of~\eqref{Winkler} by conditioning on $W = w$, and we need only show that conditional inequality holds with equality if and only if 
$w \in \left\{ - \frac{r + 1}{r}, 0, \frac{r + 1}{r} \right\}$ .  If $w = 0$, then
\begin{equation}
y^r \P(|U W| > y\,|\,W = w) = 0 = \left( \frac{r}{r + 1} \right)^r \E (|U W|^r\,|\,W = w). 
\end{equation}
Otherwise,
\begin{equation}
y^r \P(|U W| > y\,|\,W = w) 
= y^r \P\!\left( U \geq \frac{y}{|w|} \right)
=
\begin{cases}
y^r \left( 1 - \frac{y}{|w|} \right), &\mbox{if $|w| > y$} \\
0, &\mbox{otherwise},
\end{cases} 
\end{equation}
while
\begin{equation}
\left( \frac{r}{r + 1} \right)^r \E (|U W|^r\,|\,W = w) = |w|^r \frac{r^r}{(r + 1)^{r + 1}}.
\end{equation}
So we need only show for $v \in (0, 1)$ that
\begin{equation}
\label{vineq}
v^r (1 - v) \leq \frac{r^r}{(r + 1)^{r + 1}},
\end{equation}
with equality if and only if $v = r / (r + 1)$.  But the left side of~\eqref{vineq} is strictly log-concave and is uniquely maximized at $v = r / (r + 1)$, with equality then holding at~\eqref{vineq}. 
\end{proof}

\begin{theorem}
\label{T:unimodal}
Let 
$r \in (0, \infty)$.  Consider 
RWMH
with a unimodal target density~$\pi$ that has a mode at some point $m \in \bbR$ and a proposal step-density that is symmetric  (about~$0$).  If  $X_0 \in L^r$, then
\begin{equation}
\label{SRWMHunimodalr}
\E |X_1 - X_0|^r < \left( \frac{2 r}{r + 1} \right)^r \E |X_0 - m|^r.
\end{equation}
\end{theorem}

\begin{proof}
It is clear that we may suppose $m = 0$ without loss of generality.  Then, 
constructing~$\bpi$ from~\eqref{bpi} as in \refS{SS:symmetrization}, we find that~$\bpi$ has mean zero and $r$th absolute moment $\E |X_0|^r$ and is symmetric and unimodal, so that by~\eqref{r bar versus not} and \refT{T:rsymmuni}, we find for the RWMH
chain~$\bX$ with target density~$\bpi$ that
\begin{align}
\label{abc}
\E |X_1 - X_0|^r
&\leq \E |\bX_1 - \bX_0|^r
= 2 \int_0^{\infty} \phi(z) z^r \P(|\bX_0| > z / 2) \dd z.
\end{align}
But  by 
\refL{L:Winkler}
we have
\begin{equation}\label{def}
z^r \P(|\bX_0| > z / 2) \leq \left( \frac{2 r}{r + 1} \right)^r \E |\bX_0|^r = \left( \frac{2 r}{r + 1} \right)^r \E |X_0|^r,
\end{equation}
where, moreover, we can have equality for at most one $z > 0$. 
Hence,  \eqref{abc}--\eqref{def} yield
\begin{equation}
\E |X_1 - X_0|^r < \left( \frac{2 r}{r + 1} \right)^r \E |X_0|^r = \left( \frac{2 r}{r + 1} \right)^r \E |X_0 - m|^r.
\end{equation}
Therefore, \eqref{SRWMHunimodalr} holds, as desired.
\end{proof}

\begin{remark}
Suppose that~$\phi$ is symmetric.
When the unimodality condition of \refT{T:unimodal} is met, this theorem is always an improvement on the more general \refT{T:main1}.  For one thing, \refT{T:unimodal} allows more values of~$r$.  For another, even for $r \in [1, \infty)$, the bound $\left( \frac{2 r}{r + 1} \right)^r$ in \refT{T:unimodal} is no larger than the bound $2^{r - 1}$ in \refT{T:main1}, and is strictly smaller except for $r = 1$.  Indeed, the ratio 
$\left( \frac{2 r}{r + 1} \right)^r / 2^{r - 1} = 2 \left( \frac{r}{r + 1} \right)^r$ is strictly log-convex, equals $1$ at $r=1$, and converges  to $2 / e < 1$ as $r \to \infty$.
\end{remark}

\begin{corollary}
\label{C:unimodal}
Consider RWMH
with a unimodal target density~$\pi$ that has a mean and mode both equal to some point $\mu \in \bbR$ and a proposal step-density that is symmetric  (about~$0$).  If $X_0 \in L^2$, then 
\begin{equation}
\label{corr1/9}
\Corr(X_0, X_1) > \tfrac19.
\end{equation}
\end{corollary}

\begin{proof}
This follows from~\eqref{Pythagoras} and \refT{T:unimodal} (with $m = \mu$ and $r = 2$).
\end{proof}

\begin{corollary}
\label{C:symmunimodal}
Suppose that~$\pi$ is symmetric and unimodal about some point $\mu \in \bbR$, and that the proposal step-density is symmetric about~$0$.  If $X_0 \in L^2$, then 
\begin{equation}
\label{corr1/9symm}
\Corr(X_0, X_1) > \tfrac19.
\end{equation}
\end{corollary}

For something of a discrete analogue of \refC{C:symmunimodal}, see \refT{T:Isymmunimodal}.

\begin{remark}
\label{R:sharp}
The  
upper bound~\eqref{SRWMHunimodalr} is sharp, in the sense that (assuming $m = 0$), for each fixed $r \in (0, \infty)$, the supremum of $\E |X_1 - X_0|^r / \E |X_0|^r$ overall all choices of~$\pi$ and~$\phi$ as in \refT{T:unimodal} equals $\left( \frac{2 r}{r + 1} \right)^r$.  [The lower bound~\eqref{corr1/9symm} is correspondingly sharp.]
To see this (which results from the sharpness in \refL{L:Winkler}), 
note that if~$\pi$ is the Unif$(- (r + 1), r + 1)$ density with distribution function~$\Pi$ and $Z = \pm 2 r$ with probability $1/2$ each, then $\E |X_0|^r = (r + 1)^{r - 1}$ and (by \refT{T:rsymmuni})
\begin{equation}
\E |X_1 - X_0|^r = (2 r)^r \P(|X_0| > r) = (2 r)^r / (r + 1).
\end{equation}
The random variable $Z$ is discrete and thus not allowed as a 
proposal step-distribution in our setting, but we can approximate it and
take $\phi \equiv \phi_{\eps}$ to be the $(1/2, 1/2)$-mixture of the uniform distributions on 
$(- 2 r - \eps, - 2 r + \eps)$ and $(2 r - \eps, 2 r + \eps)$; then 
$\E |X_1 - X_0|^r / \E |X_0|^r$ will converge to $\left( \frac{2 r}{r + 1} \right)^r$ as $\eps \downarrow 0$ by 
\eqref{rsymmuni}, the above calculations, and the bounded convergence theorem.
However,
as stated in \refT{T:unimodal},
this supremum [and, correspondingly likewise, the infimum of $1/9$ in \eqref{corr1/9symm}] cannot be achieved (with a continuous proposal step-distribution),
as a consequence of
the necessary and sufficient condition 
for equality in \refL{L:Winkler}.
\end{remark}

\begin{example}
\label{E:normal}

(a)~Recalling
the discussion at~\eqref{ratio}, \refT{T:unimodal} with $r \neq 2$ can at least sometimes be used to bolster the evidence provided by \refC{C:symmunimodal} of positive dependence between $X_0$ and $X_1$.  For example, suppose that $X_0$ and $X_0^*$ are independent standard normal random variables.  Then, by \refT{T:unimodal} with $m = 0$, we  have
\begin{equation}
\label{UBratio}
\rho(r) := \frac{\E |X_1 - X_0|^r}{\E |X_0^* - X_0|^r} \leq \left( \frac{\sqrt{2}\,r}{r + 1} \right)^r =: \hrho(r),
\qquad r > 0;
\end{equation}
the bound $\hrho(r)$ here is log-convex in~$r$, with limit~$1$ as $r \downarrow 0$, minimum value approximately $0.6834$ at $r \approx 0.6145$, value~$1$ at $r = \sqrt{2} + 1$, and asymptotic growth 
$\hrho(r) \sim e^{-1} 2^{r / 2}$ as $r \to \infty$.  In particular, $\hrho(r) < 1$ (indicating positive dependence) if (and only if) $r < \sqrt{2} + 1$.

(b)~Of course, the bound 
$\hrho(r)$ in~\eqref{UBratio}
is just that:\ a bound, independent of the choice of step-density~$\phi$.  If, for example, 
$\phi$ is taken to be the normal density with mean~$0$ and variance~$4$, then 
by \refT{T:rsymmuni}, with~$Y$ and $Y^*$ independent standard normal random variables, 
the numerator of $\rho(r)$ equals
\begin{equation}
\label{rhornumerator}
\E |X_1 - X_0|^r
= 2^r \E[|Y|^r {\bf 1}(|Y| < |Y^*|)]
= 2^r \E \E \left[ |Y|^r {\bf 1}(|Y| < |Y^*|)\,\Big|\,|Y| \right].
\end{equation}
The expression~\eqref{rhornumerator} can be evaluated using a hypergeometric
function, but this does not lead
to a particularly simple
form for general~$r$; 
however $\rho(r)$ can be evaluated numerically (and exactly for integer $r$),
and with simple arguments we can bound~$\rho(r)$ more sharply than in
part~(a) of this example, as follows. 
With~$\varphi$ (not to be confused with~$\phi$) denoting the standard normal density, for 
$y > 0$ we have, by a well-known tail bound,
\begin{equation}
\label{normal tail bound}
\E \left[ |Y|^r {\bf 1}(|Y| < |Y^*|)\,\Big|\,|Y| = y \right]
= y^r \P(|Y^*| > y)
< 2\,y^{r - 1} \varphi(y),
\end{equation}
and so
\begin{align}
  \E |X_1 - X_0|^r 
< 2^{r + 1} \E[|Y|^{r - 1} \varphi(|Y|)]
= \frac{2^{(r + 1) / 2}}{\sqrt{\pi}} \E |Y|^{r - 1}.
\end{align}
The denominator of $\rho(r)$ equals $2^{r / 2} \E |Y|^r$, so
\begin{align}
\rho(r) 
&< \sqrt{\frac{2}{\pi}}\,\frac{\E |Y|^{r - 1}}{\E |Y|^r}
= \sqrt{\frac{2}{\pi}}\,\frac{2^{(r - 1) / 2} \pi^{-1/2} \Gamma(r / 2)}{2^{r / 2} \pi^{-1/2} \Gamma((r + 1) / 2)} 
\notag \\
&= \pi^{-1/2}\,\frac{\Gamma(r / 2)}{\Gamma((r + 1) / 2)} =: \trho(r)
\label{trho}
\end{align}
The function~$\trho$ is strictly decreasing and $\trho(1) = 1$; using this [unlike using~\eqref{UBratio}], we can conclude $\rho(r) < 1$ for all $r \in [1, \infty)$.  Further, since the inequality in~\eqref{normal tail bound} is also asymptotic equality as $y \to \infty$, we have $\rho(r) \sim \trho(r) \sim \sqrt{2 / (\pi r)}$ as $r \to \infty$.

By a different argument, we also have $\rho(r) < 1$ for $r \in (0, 1)$, and therefore $\rho(r) < 1$ for all 
$r \in (0, \infty)$.  Here is an argument that works for $r \in (0, 2)$ (and, pushing just a little more, also  for $r = 2$).  Return to the first of the two expressions for $\E |X_1 - X_0|^r$ in~\eqref{rhornumerator} and condition now on $|Y^*|$ instead of $|Y|$.  Since, for fixed $y^* > 0$ the function $y^r$ is increasing in $y \in (0, \infty)$ and the function ${\bf 1}(y < y^*)$ is decreasing, applying ``Chebyshev's other inequality''~\cite{FJ(1984)} 
shows that, conditionally given $Y^*$, 
the variables $|Y|^r$ and ${\bf 1}(|Y|<Y^*)$ are
negatively correlated.
Then taking expectations yields
\begin{align}
\E |X_1 - X_0|^r \leq 
2^r\E|Y|^r\P(|Y|<|Y^*|)=
2^{r - 1} \E |Y|^r
\end{align} 
and hence
\begin{equation}
\label{small r bound}
\rho(r) \leq 2^{(r - 2) / 2} < 1.
\end{equation}
It is easy to see 
that $\rho(r) \to \P(X_1\neq X_0)=1/2$ 
as $r \downarrow 0$, so the bound $2^{(r - 2) / 2}$ is sharp in this limit.

Although we shall not attempt a proof, plots indicate that $\rho(r)$ is strictly decreasing in $r \in (0, \infty)$.

For 
this example, not only is it true that $\rho(r) < 1$, but in fact we have the even stronger evidence of positive dependence between $X_0$ and $X_1$ that
\begin{equation}
\E h(|X_1 - X_0|) < \E h(|X^*_0 - X_0|)
\end{equation}
for any strictly increasing function $h:[0, \infty) \to \bbR$.  This follows from the stochastic inequality between 
$|X_1 - X_0|$ and $|X^*_0 - X_0|$, which in turn follows from the decreasing ratio of their densities over 
$(0, \infty)$.
\end{example}

\subsection{Higher dimensions} \label{SS:high}

\refT{T:unimodal} and \refC{C:symmunimodal} can be extended to higher dimensions.  {\bf In this subsection we assume that the target density~$\pi$ satisfies
\begin{equation}
\label{symmstarunimodal}
\mbox{$\pi(\xx) = \hpi(\| \xx \|)$ for some nonincreasing function~$\hpi$ on $[0, \infty)$}
\end{equation}
and that the proposal step-density is symmetric  (about~$\zzero$).}
This is, assuming that the target distribution has a density (and that we
choose it suitably), 
equivalent to the condition that the target distribution is spherically symmetric and 
star unimodal
(slightly extending \cite[Criterion in Section~2.2]{DJ-D(1988)}); 
if $d = 1$, 
the condition is equivalent to the hypothesis of~\refC{C:symmunimodal} with $\mu = 0$.

We now extend the bounds in~\eqref{SRWMHunimodalr} and \eqref{corr1/9}
to this setting.

\begin{theorem}
\label{T:unimodal-d}
Let 
$r \in (0, \infty)$.  Consider 
RWMH
with a  target density~$\pi$ that satifies \eqref{symmstarunimodal}
and a proposal step-density that is symmetric  (about~$0$).  If\/ $\XX_0 \in L^r$, then
\begin{equation}
\E \| \XX_1 - \XX_0 \|^r < \left( \frac{2 r}{r + 1} \right)^r \E \| \XX_0 \|^r.
\end{equation}
\end{theorem}
\begin{proof}
By \refT{T:unimodal}, we may suppose that $d \geq 2$.
First note that by \eqref{symmstarunimodal} we have
\begin{equation}
\min\{\pi(\xx), \pi(\yy)\} = \hpi(\max\{\| \xx \|, \| \yy \|\}),
\end{equation}
and thus for $\zz \neq 0$ we have
\begin{align}
1 - \dtv(\XX_0, \XX_0 + \zz) 
&= \int\!\hpi(\max\{\| \xx \|, \| \xx - \zz \|\}) \dd \xx \nonumber \\
&= \P\!\left( \left| \frac{\langle \XX_0, \zz \rangle}{\| \zz \|} \right| > \frac12 \| \zz \| \right)
= \P(|X_{01}| > \tfrac12 \| \zz \|), \label{TVhigh}
\end{align}
where $X_{01}$ is the first coordinate of $\XX_0$ and has a unimodal distribution.  Again by Winkler's  inequality
(\refL{L:Winkler}),
\begin{equation}
\label{morehigh}
\| \zz \|^r \P(|X_{01}| > \tfrac12 \| \zz \|) 
\leq \left( \frac{2 r}{r + 1}\right)^r \E |X_{01}|^r 
< \left( \frac{2 r}{r + 1}\right)^r  \E \| \XX_0 \|^r,
\end{equation}
where the second inequality is strict because $d \geq 2$ and $\XX_0$ has a density.  (We don't even need to consider strictness of the first inequality.)
Combining~\eqref{SRWMHsymmphir} and~\eqref{TVhigh}--\eqref{morehigh} gives the desired conclusion.
\end{proof}

\begin{corollary}
\label{C:unimodal-d}
Consider 
RWMH
with a  target density~$\pi$ that satifies \eqref{symmstarunimodal}
and a proposal step-density that is symmetric  (about~$0$).  If $\XX_0 \in L^2$, then
\begin{equation}
\label{corr1/9-d}
\trCorr(\XX_0, \XX_1) > \tfrac19.
\end{equation}
\nopf
\end{corollary}

\begin{remark}
In the setting of this subsection, we see from symmetry that if $\XX_0 \in L^2$, then
\begin{equation}
\trCov(\XX_0, \XX_1) = d \Cov(X_{01}, X_{11})
\end{equation}
and
\begin{equation}
\trVar \XX_0 = d \Var X_{01},
\end{equation}
so that
\begin{equation}
\trCorr(\XX_0, \XX_1) = \Corr(X_{01}, X_{11}).
\end{equation}
More generally, we conclude for any nonzero vector~$\ccc$  that
\begin{equation}
\Corr(\langle \ccc, \XX_0 \rangle, \langle \ccc, \XX_1 \rangle) > \tfrac19.
\end{equation}
\end{remark}

\section{Integer-grid Metropolis--Hastings} \label{S:I}

As a consequence of \refT{T:TVLBZ}, all results in Sections \ref{S:setup}--\ref{S:poslin} hold, 
\emph{mutatis mutandis}, also for IMH in $\bbZ^d$, except 
that we only get weak inequalities in \refT{T:main} and \refC{C:maincov}.
(Extension  of \refS{S:unimodaltarget}, even in dimension $d = 1$, is much more delicate, as discussed later in this section.)
Indeed, the results in Sections \ref{S:setup}--\ref{S:poslin} have the same proofs for $\bbZ^d$ as for 
$\bbR^d$, replacing Lebesgue measure by counting measure, i.e.,\ replacing
integrals by sums, and using \refT{T:TVLBZ} instead of \refT{T:TVLB}.
For completeness, we state the analogues of \refT{T:main} and \refC{C:maincov} explicitly, with IRWMH our abbreviation for integer-grid random walk Metropolis--Hastings, that is, IMH as in \refR{R:discrete set-up} with $q(\xx, \yy)$ a function of $\yy - \xx$.

\begin{theorem}
\label{T:Imain}
Let $r \in [2, \infty)$.  For IRWMH,
if $\XX_0 \in L^r$ 
then for any $\mm\in\bbR^d$ we  have
\begin{equation}
\label{ZrUB}
\E \| \XX_1 - \XX_0 \|^r \leq 2^{r - 1} \E \| \XX_0 - \mm \|^r.
\end{equation}
\nopf
\end{theorem}

\begin{corollary}
\label{C:Imaincov}
For 
RWMH, if $\XX_0 \in L^2$ then
\begin{equation}
\label{ZLB}
\trCorr(\XX_0, \XX_1) \geq 0.
\end{equation}
\nopf
\end{corollary}

The following example shows that in general we cannot improve from weak to strict inequalities in \refT{T:Imain} and \refC{C:Imaincov}.
 
\begin{example}\label{E:discrete}
For a class of one-dimensional IRWMH examples where equality holds, suppose
that $X_0 \in L^2$ takes values in~$\bbZ$, and without loss of generality
assume $\E X_0 \in (0, 1]$.  Suppose also that the distribution $\cL(Z)$ of
the proposal transition-steps is concentrated on $\{-1, 0, 1\}$.
Take $m:=\E X_0$.  
\refL{L:TVLB equally spaced} and \refR{R:TVLB+} then apply directly, taking
$r := 2$, $b := - \E X_0$, and $p_n := \P(X_0 = n)$ for $n \in \bbZ$.  It
follows that,
for this class of examples and
excluding the trivial case $X_0 = 1$ \as,\ we have equality in~\eqref{ZrUB} [\ie,\ in~\eqref{ZLB}] if and only if $X_0$ is concentrated on $\{0, 1\}$
and~$Z$ is concentrated on $\{-1, 1\}$ 
with $Z = 2 X_0 - 1$ in distribution; in this exceptional case, 
$X_1 = (Z + 1) / 2$, which (by construction) is an independent (and hence uncorrelated) copy of $X_0$.
\end{example}

{\bf For the remainder of this section we assume about our IRWMH on $\bbZ^d$ that 
$d = 1$; the target distribution, with probability mass function~$\pi$, is discrete unimodal about~$0$ }(this centering being with no loss of generality, and our assumption then equivalent to having 
$\pi_n \geq \pi_{n - 1}$ for $n \leq 0$ and $\pi_n \leq \pi_{n - 1}$ for $n \geq 1$){\bf , and the probability mass function~$\phi$ for the proposal step-distribution is symmetric about~$0$}.  Since \refP{P:r bar versus not} carries over to the IRWMH setting, just as in \refS{S:unimodaltarget} for simplicity we may as well, and do, {\bf assume that~$\pi$ is symmetric about~$0$}. 

As the following example (closely related to \refR{R=} and \refE{E:discrete}) shows, the IRWMH analogues of \refT{T:unimodal} and \refC{C:symmunimodal} do \emph{not} hold, and in fact we cannot do better than the general \refT{T:Imain} even with the special assumptions we have just set forth.

\begin{example}
\label{E:bad discrete unimodal}
Let $d=1$, let $X_0$ be uniformly distributed on $\{-1, 0, 1\}$, and let~$Z$ be uniformly distributed on 
$\{-2, 2\}$.  Then  $\P(|X_1 - X_0| = 2) = 1/3 = 1 - \P(|X_1 - X_0| = 0)$, and so
\begin{equation}
\E |X_1 - X_0|^r = \tfrac13 2^r = 2^{r - 1} \times \tfrac23 = 2^{r - 1} \E |X_0|^r.  
\end{equation}
\end{example} 

In light of \refE{E:bad discrete unimodal}, it is perhaps surprising that there is an IRWMH analogue of \refC{C:symmunimodal} if we place an additional parity constraint on the support of~$\phi$.  

\begin{theorem}
\label{T:Isymmunimodal}
In the IRWMH setting, suppose that~$\pi$ is symmetric and unimodal about~$0$, but not degenerate at~$0$, and that the proposal step-distribution (say, of the random variable~$Z$) is symmetric and supported on the odd integers.  If $X_0 \in L^2$, then
\begin{equation}
\label{Icorr1/10symm}
\Corr(X_0, X_1) \geq \tfrac{1}{10},
\end{equation}
with equality if and only if the distribution of $X_0$ is a mixture of unit mass at~$0$ and 
{\rm Unif}$\{-2, -1, 0, 1, 2\}$ and $Z = \pm 3$ with probability $1/2$ each. 
\end{theorem}

\begin{proof}
The assertion~\eqref{Icorr1/10symm} is equivalent to
\begin{equation}
\label{Icorr1/10symm recast}
\E (X_1 - X_0)^2 \leq \tfrac95 \E X_0^2,
\end{equation}
and this is a trivial equality if $X_0 \equiv 0$.
Next, if $X_0$ and~$Z$ are distributed as described following display~\eqref{Icorr1/10symm}, we may (by conditioning on the mixing variable) suppose that either $X_0 \equiv 0$ (already considered) or that $X_0 \sim \mbox{Unif}\{-2, -1, 0, 1, 2\}$, in which case
$\P(|X_1 - X_0| = 3) = \frac25 = 1 - \P(|X_1 - X_0| = 0)$, so that
\begin{equation}
\E (X_1 - X_0)^2 = \tfrac{18}{5} = \tfrac95 \times 2 = \tfrac95 \E X_0^2 = \tfrac95 \Var X_0.
\end{equation}
We will complete the proof by showing that
\begin{equation}
\label{Icorr1/10symm strict}
\E (X_1 - X_0)^2 < \tfrac95 \E X_0^2
\end{equation}
for all other distributional choices of $X_0 \in L^2$ and~$Z$.

The key to proving~\eqref{Icorr1/10symm strict} is \cite[Theorem~4.3]{DJ-D(1988)}.  It was noted in the proof of \refL{L:Winkler} that a distribution is unimodal about~$0$ if and only if it is a mixture of unilateral uniform distributions, and thus (by conditioning on the mixing variable) we needed only to consider unilateral uniform distributions.  Somewhat similarly, \cite[Theorem~4.1]{DJ-D(1988)} asserts that a distribution on~$\bbZ$ is discrete unimodal about~$0$ if and only if it is a mixture of bilateral discrete uniform distributions, where a bilateral discrete uniform distribution is a uniform distribution on a set $\{- j, \ldots, k\}$ with $j \geq 0$ and $k \geq 0$.  Moreover, \cite[Theorem~4.3]{DJ-D(1988)} asserts that a \emph{symmetric} distribution on~$\bbZ$ is discrete unimodal about~$0$ if and only if it is a mixture of Unif$\{-k, \ldots, k\}$ distributions with $k \geq 0$.  By conditioning on the mixing variable, we thus need only to prove~\eqref{Icorr1/10symm strict} (i)~when $X_0 \sim \mbox{Unif$\{-k, \ldots, k\}$}$ and $k \geq 1$ and $k \neq 2$ (for general symmetric~$Z$) and (ii)~when $X_0 \sim \mbox{Unif$\{-2, -1, 0, 1, 2\}$}$ and~$Z$ is symmetric with $\P(|Z| = 3) < 1$.

For~$\phi$ supported on the odd integers, the following discrete analogue of \refT{T:rsymmuni} follows, by the same proof, \emph{mutatis mutandi}:
\begin{align}
\E |X_1 - X_0|^r 
&= 2 \sum_{z = 1}^{\infty} \phi(z) z^r \P(|X_0| > z / 2) \nonumber \\
&= 2 \sum_{y = 0}^{\infty} \phi(2 y + 1) (2 y + 1)^r \P(|X_0| > y).
\label{rsymmuni discrete}
\end{align}
So, with $r = 2$, we need only show
\begin{equation}
(2 y + 1)^2 \P(|U_k| > y) \leq \tfrac95 \E U_k^2, 
\end{equation}
with an appropriate demonstration of strict inequalities, when $U_k$ is distributed 
$\mbox{Unif}\{-k, \ldots, k\}$ and~$y$ is a nonnegative integer.  Since $|U_k| \leq k$ \as,\ 
$\P(|U_k| > y) = 2 (k - y) / (2 k + 1)$ for  $y \in \{0, \ldots, k - 1\}$,
and 
\begin{equation}
\E U_k^2 
= \frac{2}{2 k + 1} \sum_{i = 1}^k i^2 
= \frac{2}{2 k + 1} \times \frac16 k (k + 1) (2 k + 1) 
= \frac13 k (k + 1),
\end{equation}
we need only show  that
\begin{equation}
\label{ky ineq}
(2 y + 1)^2 (k - y) \leq \tfrac{3}{10} k (k + 1) (2 k + 1)\mbox{\ for $k, y \in \bbZ$ with $0 \leq y \leq k - 1$},
\end{equation} 
with strict inequality unless $k = 2$ and $y = 1$.

To prove~\eqref{ky ineq}, note that, as a function of real $y \in [0, \infty)$, the \lhs{} is strictly log-concave with maximum at $y = (4 k - 1) / 6$ and maximum value $2 (2 k + 1)^3 / 27$.  Now it is easy to check that
\begin{equation}
\frac{\frac{2}{27} (2 k + 1)^3}{\frac{3}{10} k (k + 1) (2 k + 1)}
= \frac{20}{81} \frac{(2 k + 1)^2}{k (k + 1)}
\end{equation}
is strictly less than~$1$ for $k \geq 5$.  It equals~$1$ for $k = 4$, but then $(4 k - 1) / 6 = 5 / 2$ is not an integer, so we have verified strict inequality in~\eqref{ky ineq} for all $k \geq 4$.  One can then check case-by-case that we have strict inequality in~\eqref{ky ineq} for $k = 1, 2, 3$, except when $k = 2$ and $y = 1$.    
\end{proof}

The inequality \eqref{Icorr1/10symm} in
\refT{T:Isymmunimodal}
is equivalent to the assertion for $r = 2$ that the supremum [call it $s(r)$] of
$\E |X_1 - X_0|^r / \E |X_0|^r$ 
taken over symmetric discrete unimodal distributions for $X_0 \in L^r$ 
and proposal transition 
distributions concentrated on the odd integers equals $9 / 5$.  We study 
the case of general $r \in (0, \infty)$ in \refApp{A:IRWMHunimodalr}.  The
treatment there is elementary but far from trivial, and we find the result
$s(r) \equiv 1$ in the special case $r \in (0, 1]$ a bit surprising.  In \refApp{A:IRWMHunimodalr}
we also generalize (somewhat) to higher dimensions. 

\begin{remark}
Our emphasis in this paper is on moment bounds that are independent of the
choice of proposal step-distribution.  However, it is perhaps worth noting a
simple special case of~\eqref{rsymmuni discrete}.  If~$\phi(-1) = \phi(1) =
\tfrac12$, then for every $r \in (0, \infty)$ we have $\E |X_1 - X_0|^r =
\P(|X_0| \geq 1) \leq \E |X_0|^r$ by Markov's inequality, with equality if
the (nondegenerate symmetric discrete unimodal) target~$\pi$ satisfies
$\pi(0) \in [1/3, 1)$ and $\pi(-1) = \pi(1) =(1 - \pi(0)) / 2$.
\end{remark}

\section{Even-order lags}\label{S:even}

If we summarize the foregoing results for $\XX_0 \in L^2$ as saying that the stationary unit-lag autocorrelation 
of~$\XX$ is nonnegative, then a \emph{much} stronger result holds (i) for even-order lags and (ii) (by the same proof) for arbitrary lags for chains in continuous time.  We state our result for the first of these cases.

\begin{theorem}
\label{T:even}
Let $X = (X_t)_{t \geq 0}$ be \emph{any} discrete-time 
stationary reversible Markov chain on \emph{any} (measurable) state space~$S$, 
and let~$\ff$ be \emph{any} (measurable) function from~$S$ to $\bbR^d$ such that $\ff(X_0) \in L^2$.  Then
\begin{align}
\trCov((\ff(X_0), \ff(X_t)) \geq 0  
\end{align}
for any even value of $t \geq 0$. 
\end{theorem}

\begin{remark}
(a)~In particular, if $S = \bbR^d$ and $\cL(\XX_0)$ is any Borel probability measure on~$\bbR^d$ having finite variance, then by taking~$\ff$ to be the identity function in \refT{T:even} we see that $\trCov(\XX_0, \XX_t) \geq 0$ for any even value of $t \geq 0$.

(b)~Since any MH chain is reversible, \refT{T:even} applies to it.  In particular, if $S = \bbR^d$, then part~(a) of this remark implies that if $\XX_0 \in L^2$, then $\trCov(\XX_0, \XX_t) \geq 0$ for any even value of $t \geq 0$.  
\end{remark}

\begin{proof}[Proof of \refT{T:even}]
The result of course holds for $t = 0$.  Further, for any even value $t_0 \geq 2$, the ``skeleton'' ($t_0 / 2$)-step stochastic process $(X_{(t_0 / 2)\,t})_{t \geq 0}$ is also a reversible Markov chain.  So it suffices to prove \refT{T:even} for $t = 2$.

For that, we use the law of total trace-covariance, conditioning on $X_1$, to compute
\begin{align}
\lefteqn{\trCov((\ff(X_0), \ff(X_2))} \nonumber \\ 
&= \E \trCov(\ff(X_0), \ff(X_2)\,|\,X_1) + \trCov(\E[\ff(X_0)\,|\,X_1], \E[\ff(X_2)\,|\,X_1]). \label{2terms}
\end{align}
Because~$X$ is a Markov chain, $X_0$ and $X_2$ are conditionally independent given $X_1$, so the first term in~\eqref{2terms} equals $\E 0 = 0$.  Because~$X$ is stationary and reversible, 
$\E[\ff(X_0)\,|\,X_1] = \E[\ff(X_2)\,|\,X_1]$ (almost surely), and so the second term in~\eqref{2terms} equals $\trVar(\E[\ff(X_0)\,|\,X_1]) = \trVar(\E[\ff(X_1)\,|\,X_0])$.  Thus
\begin{equation}
\label{coveven}
\trCov((\ff(X_0), \ff(X_2)) = \trVar(\E[\ff(X_1)\,|\,X_0]) \geq 0, 
\end{equation}
as desired. 
\end{proof}

\begin{remark}
An alternative proof of \refT{T:even} uses the spectral decomposition of the symmetric operators (on the $L^2$-space with respect to the reversing measure) that form the semigroup for the reversible Markov chain, noting that discrete even times and continuous times both yield positive semidefinite operators.  This alternative proof can also be used to yield the nonnegative-covariance result of \refT{T:even} at odd times, too, provided that the given Markov chain~$X$ has 
$\P(X_{t + 1} = x\,|\,X_t = x) \geq 1/2$ for every $x \in S$.
This can
always be arranged if one is willing to switch attention from the given
chain~$X$ to the ``lazy'' version~$\tX$ of the chain which averages the
one-step transition kernel for~$X$ with the identity kernel. 
\end{remark}

\begin{remark}
Let~$\XX$ be any 
RWMH
chain with the target and proposal distribution both having densities and with 
$\XX_0 \in L^2$.  By Theorems~\ref{T:main} and~\ref{T:even} we then 
have $\trCov(\XX_0, \XX_t) \geq 0$ for $t = 0, 1, 2, 4, 6, 8, \ldots$
(with strict inequality at least for $t = 0, 1$).
Nonnegativity of $\trCov(\XX_0, \XX_t)$ for $t = 3, 5, 7, \ldots$ (without switching to the lazy chain, of course) remains an open question.   
\end{remark}

\appendix

\section{A generalized moment lower bound on total variation distance for location shift} \label{A:TVLB}

We state here a general theorem used in the proof of our main \refT{T:main}.

\begin{theorem}
\label{T:TVLB}
Let\/ $d\ge1$ and
let\/ $\XX_0$ be a random vector in $\bbR^d$ having a density~$\pi$.
Let\/ $r \in [2, \infty)$.
If\/ $\phi \in [0, 1]$, and
$\zz$ and\/ $\mm$  are any deterministic vectors in $\bbR^d$, then  
\begin{equation}
\label{TV}
\|\zz\|^r \int\!\min\{\phi\,\pi(\xx), (1 - \phi) \pi(\xx - \zz) \} \dd \xx < 2^{r - 2} \E \|\XX_0-\mm\|^r.
\end{equation}
Furthermore, for any nonzero $\ccc\in\bbR^d$ and $m\in\bbR$ we have
\begin{equation}
\label{TVc}
|\innprod{\ccc,\zz}|^r \int\!\min\{\phi\,\pi(\xx), (1 - \phi) \pi(\xx - \zz) \} \dd \xx < 2^{r - 2} \E |\innprod{\ccc,\XX_0}-m|^r.
\end{equation}
\end{theorem}

We postpone the proof and comment first on the special case $\phi=\frac12$.
In this case, by \eqref{dtv} [see also \eqref{dtv2}], the integral in
\eqref{TV} equals $\frac12[1 - \dtv(\XX_0, \XX_0 + \zz)]$, and we thus
obtain
\begin{corollary}\label{C:TVLB}
Let\/ $\XX_0$ be a random vector in $\bbR^d$ having a density,
and 
let\/ $r \in [2, \infty)$.
If\/
$\zz$ and\/ $\mm$  are any deterministic vectors in $\bbR^d$, then  
\begin{equation}
\label{TVwithoutphi}
\|\zz\|^r [1 - \dtv(\XX_0, \XX_0 + \zz)] < 2^{r - 1} \E \|\XX_0-\mm\|^r.
\end{equation}
Furthermore, for any nonzero $\ccc\in\bbR^d$ and $m\in\bbR$ we have
\begin{equation}
\label{TVCc}
|\innprod{\ccc,\zz}|^r [1 - \dtv(\XX_0, \XX_0 + \zz)] < 2^{r - 1} 
\E |\innprod{\ccc,\XX_0}-m|^r.
\end{equation}
\end{corollary}

Since we are mainly interested in the applications above to the MH setting
in \refS{SS:setting}, we have stated these results for a random vector
$\XX_0$ having a  density (in other words, for absolutely continuous random
vectors).
However, we will also show in \refApp{A:more TV} (see \refT{T:TVLB3})
that if we relax $<$  to $\leq$ in \eqref{TVwithoutphi}--\eqref{TVCc}, then
the results extend not only to arbitrary random vectors in $\bbR^d$ but in fact to random variables taking values in any Banach space, and also from $r \geq 2$ to $r \geq 1$.

We believe the results of \refT{T:TVLB}, \refC{C:TVLB}, and \refT{T:TVLB3}
to be 
 new (even in dimension one)
and perhaps of independent interest. 
(They also motivate the title of this appendix.)

To prove 
\refT{T:TVLB},
we will first state and prove a lemma (\refL{L:TVLB equally spaced}) 
which is essentially a discrete version of the case $d=1$ of \refT{T:TVLB}.

\begin{lemma}
\label{L:TVLB equally spaced}
Let $r \in [2, \infty)$ and $\phi \in [0, 1]$.  For any doubly infinite
nonnegative sequence $(p_n)_{n \in \bbZ}$ 
and any $b \in \bbR$, we have
\begin{equation}
\label{pineq2}
\sum_n \min\{(1 - \phi) p_{n - 1}, \phi p_n\} 
\le
2^{r - 2} \sum_n |b + n|^r p_n.
\end{equation}
If further,
for simplicity, we assume that 
$b \in [0, 1)$, then equality 
holds in \eqref{pineq2}
if and only if 
either both sides are infinite or
one of the following (non-exclusive) sets of
conditions holds:
\begin{romenumerate}
\item \label{Ltvlb0}
$p_n = 0$ for all $n$;
\smallskip
\item \label{Ltvlb1}
\begin{alphiienumerate}
\item $p_n = 0$ for all $n\neq0$, and
\item $b=0$;
\end{alphiienumerate}
\smallskip
\item\label{Ltvlb2}
\begin{alphiienumerate}
\item $p_n=0$  for $n \notin \{-1, 0\}$, 
\item $p_{-1} = p_0 $, and
\item 
$\phi = b = 1/2$; 
\end{alphiienumerate}
\smallskip
\item\label{Ltvlb3}
\begin{alphiienumerate}
\item $p_n=0$  for $n \notin \{-1, 0\}$,
\item $r=2$, and
\item 
$\phi = b = \frac{p_{-1}}{p_{-1}+ p_0} $.
\end{alphiienumerate}
\end{romenumerate}
\end{lemma}

\begin{remark}\label{R:TVLB+}
The assumption $b\in[0,1)$ in the second part of \refL{L:TVLB equally spaced}
is for convenience in the statements of the exceptional cases.
The conditions for equality for general $b$ are given by 
the conditions above applied to the fractional part $b-\floor{b}$ and the
shifted sequence $(p_{n-\floor{b}})_n$.
\end{remark}

\begin{proof}[Proof of \refL{L:TVLB equally spaced}]
We assume for convenience $b\in[0,1)$ in the proof. 
The general case then follows 
by applying \eqref{pineq2} to the fractional part $b-\floor{b}$ and the
shifted sequence $(p_{n-\floor{b}})_n$.

We assume
$\sum_n |b + n|^r p_n < \infty$, since otherwise the result holds trivially;
note that then also $\sum_np_n<\infty$, and thus both sides of
\eqref{pineq2} are finite.
Hence
we may rewrite \eqref{pineq2} as
\begin{equation}
\label{pineq1}
2^{r - 2} \sum_n |b + n|^r p_n - \sum_n \min\{(1 - \phi) p_{n - 1}, \phi p_n\} \geq 0.
\end{equation}
If we replace a single $p_n$ by 0, keeping all $p_m$ for $m \neq n$ unchanged, then
the \lhs{} of~\eqref{pineq1} decreases by
\begin{align}\label{sa1}
\lefteqn{2^{r - 2} |b + n|^r p_n - \min\{(1- \phi) p_{n - 1}, \phi p_n\} -
  \min\{(1 - \phi) p_n, \phi p_{n + 1}\}} 
\notag\\&\qquad
 \geq 2^{r - 2} |b + n|^r p_n - \phi p_n - (1 - \phi) p_n = (2^{r - 2} |b + n|^r - 1) p_n,
\end{align}
which is nonnegative provided $|b + n| \geq 1$.
Hence, if the inequality~\eqref{pineq1} holds for the modified sequence, then it
holds for the original sequence.  Therefore, if we make the replacement sequentially for the values
$n = 1, 2, -2, 3, -3, \ldots, n_0, - n_0$ with $n_0 \geq 2$, we find by letting $n_0 \to \infty$ and
invoking the dominated convergence theorem, and writing $p := p_0$ and $q := p_{-1}$,
that the \lhs{} of \eqref{pineq1} is at least
\begin{align}
\lefteqn{\hspace{-.3in}2^{r - 2} (|b - 1|^r p_{-1} + |b|^r p_0) - \min\{(1 - \phi) p_{-1}, \phi p_0\}} 
\nonumber \\
&= 2^{r - 2} [(1 - b)^r q + b^r p] - \min\{(1 - \phi) q, \phi p\}. \label{qp}   
\end{align}
It is clear that the expression~\eqref{qp} is equivariant under scaling 
by $q + p$, and it trivially vanishes when $q + p = 0$,
so to establish its nonnegativity
we may and do assume that $q + p = 1$.  Then the expression becomes
\begin{align}
\lefteqn{\hspace{-.7in}2^{r - 2} [(1 - b)^r (1 - p) + b^r p] - \min\{(1 - \phi) (1 - p), \phi p\}} \nonumber \\
&\geq 2^{r - 2} [(1 - b)^r (1 - p) + b^r p] - p (1 - p). \label{bp} 
\end{align}
As a function of~$b$, the \rhs{} of~\eqref{bp} is minimized when 
\begin{align}\label{sa2}
b = \frac{q^{1 / (r - 1)}}{q^{1 / (r - 1)} + p^{1 / (r - 1)}},  
\end{align}
giving a value of
\begin{align}
\label{ponly}
\left[ \frac{2^{r - 2}}{[q^{1 / (r - 1)} + p^{1 / (r - 1)}]^{r - 1}} - 1 \right] q\,p
&\geq \left[ \frac{2^{r - 2}}{[2^{-1 / (r - 1)} + 2^{-1 / (r - 1)}]^{r - 1}} - 1 \right] q\,p \nonumber \\
&= 0. 
\end{align}
This completes the proof of~\eqref{pineq1}, and thus  of~\eqref{pineq2}.

It is easy to verify that equality holds in \eqref{pineq2} in the cases
\ref{Ltvlb0}--\ref{Ltvlb3}. 

Conversely, suppose that equality holds in \eqref{pineq2}.
Then we cannot have strict inequality anywhere in the argument above.
Hence, $p_n=0$ for $|n|\ge2$, 
since otherwise  \eqref{sa1} would be positive.
Furthermore, also $p_1=0$, since otherwise replacing $p_1$ by $0$, as in
\eqref{sa1} and now also using $p_2=0$, would decrease the \lhs{} of
\eqref{pineq1} by
\begin{align}\label{sa3}
2^{r - 2} (b + 1)^r p_1 - \min\{(1- \phi) p_{0}, \phi p_1\}
&\ge
p_1 - \min\{(1- \phi) p_{0}, \phi p_1\}
\notag\\&
>0.
\end{align}
As above, let $p:=p_0$ and $q:=p_{-1}$.
By scaling, and excluding case \ref{Ltvlb0},
we may again assume  $q + p = 1$, and we then must have equality in
\eqref{bp} and \eqref{ponly}; we also have the equation \eqref{sa2} since
the \rhs{} of \eqref{bp} has a strict minimum at this $b$.

If $p=0$, then \eqref{sa2} yields $b=1$, which is excluded.
If $p=1$, then $q=0$,
so \eqref{sa2} yields $b=0$, and thus case \ref{Ltvlb1}.

It remains to consider the case $0<p<1$.
Then equality in \eqref{bp} holds if and only if  $\phi=1-p$.
If $r=2$, this and \eqref{sa2} yield case \ref{Ltvlb3}.

Finally, if $0<p<1$ and $r>2$, then equality in \eqref{ponly} requires
$p=q=\frac12$, 
and thus $b=\frac12=\phi$ by \eqref{sa2}
and the relation $\phi=1-p$  just shown.
This is case \ref{Ltvlb2}.
\end{proof}

\begin{remark}\label{R:TVLB-}
The 
condition $r\ge2$ is necessary for  \refL{L:TVLB equally spaced}
in general.
If $r<2$, we may produce a counterexample as follows.
Take $p_0:=p$ with $2^{r-2}<p<1$, and
let $p_{-1}:=q:=1-p$ and 
$p_n:=0$ for  $n\notin\set{-1, 0}$.
Let also $b:=0$ and $\phi:=1-p$.
Then the \lhs{} of \eqref{pineq1} is
\begin{align}\label{sd1}
2^{r-2}q-\min\{(1 - \phi) p_{ - 1}, \phi p_0\}
=2^{r-2}q-pq
=\bigpar{2^{r-2}-p}q<0.
\end{align}

This example provides also a counterexample for $r<2$ for the discrete
\refT{T:TVLBZ} below: let 
$d=1$, $\P(X_0=n)=p_n$, $z=1$, and $m=0$.

A suitable perturbation making $X_0$ absolutely continuous shows that
$r\ge2$ is necessary for \refT{T:TVLB} too.

Furthermore, taking 
$Z\in\set{-1, 1}$ random with $\P(Z=-1)=q$ and $\P(Z=1)=p$
shows that $r\ge2$ is necessary in \refT{T:Imain}
(this is an instance of \refE{E:discrete}),
and suitable perturbations of $X_0$ and $Z$ 
show that $r\ge2$ is necessary in \refT{T:main}.
\end{remark}

\begin{proof}[Proof of \refT{T:TVLB}]
We begin with \eqref{TVc}. This is trivial if $\innprod{\ccc,\zz}=0$,
since we assume that $\XX_0$ has a density and thus
$\innprod{\ccc,\XX_0}-m\neq0$ a.s.
Hence we may assume $\innprod{\ccc,\zz}\neq0$.
By dividing $\ccc$ and $m$ by $\innprod{\ccc,\zz}$, we may further assume
$\innprod{\ccc,\zz}=1$.

Let $g(\xx):=\min\bigcpar{\phi\pi(\xx),(1-\phi)\pi(\xx-\zz)}$
and
$h(\xx):=|\innprod{\ccc,\xx}-m|^r\pi(\xx)$.
Further, let  
\begin{align}\label{e99}
A:=\set{\xx\in\bbR^d:\innprod{\ccc,\xx}\in[0,1)}.  
\end{align}
Then, for $n\in\bbZ$,  the  translate
$A+n\zz$ equals $\set{\xx\in\bbR^d:\innprod{\ccc,\xx}\in[n,n+1)}$.
These translates form a partition of $\bbR^d$, and thus  for any
nonnegative measurable function~$f$ we have
\begin{align}\label{jr1}
\int_{\bbR^d}f(\xx)\dd\xx&
=\sum_n \int_{A+n\zz}f(\xx)\dd\xx
=\int_{A}\sum_n f(\xx+n\zz)\dd\xx.
\end{align}
For every~$\xx$ we have
by \refL{L:TVLB equally spaced},
using \eqref{pineq2} with 
$b:=\innprod{\ccc,\xx}-m$
and
$p_n:=\pi(\xx+n\zz)$, that
\begin{align}\label{jr2}
&  \sum_n g(\xx+n\zz)
=\sum_n \min\{(1-\phi)p_{n-1}, \phi p_n\}
\le 2^{r-2}\sum_n |b+n|^rp_n
\notag\\&\hskip4em
= 2^{r-2}\sum_n |\innprod{\ccc,\xx+n\zz}-m|^rp_n
=2^{r-2}\sum_n h(\xx+n\zz).
\end{align}
We integrate this over $A$, and, using \eqref{jr1} twice, obtain
\begin{align}\label{jr3}
\int_{\bbR^d}g(\xx)\dd\xx&
=\int_{A}\sumnz g(\xx+n\zz)\dd\xx
\le\int_A 2^{r-2}\sumnz h(\xx+n\zz)\dd\xx
\notag\\&
=2^{r-2}\int_{\xx\in\bbR^d}h(\xx) \dd\xx
=2^{r-2}\E|\innprod{\ccc,\XX_0}-m|^r
.\end{align}

This shows \eqref{TVc} with weak inequality. To rule out equality, we note 
first that $g(\xx)\le\phi\pi(\xx)$, and thus $\int g(\xx)\dd\xx<\infty$.
Hence, equality in \eqref{TVc}, and thus in \eqref{jr3}, implies equality in
\eqref{jr2} for almost every $\xx\in A$.
Furthermore, \eqref{jr1} shows that
\begin{align}\label{jrpi}
\int_{A}\sum_n \pi(\xx+n\zz)\dd\xx
=\int_{\bbR^d}\pi(\xx)\dd\xx=1,  
\end{align}
and thus 
$0<\sum_n \pi(\xx + n \zz) < \infty$ holds for $\xx$ in a subset
of $A$ of positive Lebesgue measure.
It follows from \refL{L:TVLB equally spaced} and \refR{R:TVLB+} that for 
such $\xx$, 
if we have equality in \eqref{jr2}, then
$b-\floor{b}\in\set{0,\phi}$.
But this holds only for a
countable number  of $b = \innprod{\ccc, \xx} - m$, and thus only for a set of
$\xx\in\bbR^d$ of Lebesgue measure zero. 
This contradiction shows that equality cannot hold in \eqref{TVc}.

We turn to \eqref{TV}. The case $\zz=\zzero$ is trivial, again because
$\XX_0$ has a density. If $\zz\neq\zzero$, we choose 
$\ccc:=\ee_{\zz}:=\zz/\norm{\zz}$, the unit vector parallel to~$\zz$.
Then $\innprod{\ccc,\zz}=\norm{\zz}$, and thus 
\eqref{TV} follows from \eqref{TVc}, with $m=\innprod{\ee_\zz,\mm}$, 
noting that
\begin{align}
  \label{a66}
|\innprod{\ee_\zz,\XX_0}-m|=|\innprod{\ee_\zz,\XX_0-\mm}|\le\norm{\XX_0-\mm}.
\end{align}
\end{proof}

\refC{C:TVLB} follows immediately, as noted above.

\section{The case $\bbZ^d$} \label{A:TVLBZ}

In this appendix we state and prove a discrete analogue of \refT{T:TVLB}.
Note that, unlike  for \refT{T:TVLB}, we do not claim strict inequalities, nor do strict inequalities hold in general (see \refR{R:poss equal}).

\begin{theorem}
\label{T:TVLBZ}
Let\/ $d\ge1$ and 
let\/ $\XX_0$ be a random vector in $\bbZ^d$ having probability mass function 
$\pi(\xx):=\P(\XX_0=\xx)$, $\xx\in\bbZ^d$. 
Let\/ $r \in [2, \infty)$.
If\/ $\phi \in [0, 1]$, and
$\zz\in\bbZ^d$ and\/ $\mm\in\bbR^d$  are any deterministic vectors, then  
\begin{equation}
\label{TVZ}
\|\zz\|^r \sum_{\xx\in\bbZ^d}\!\min\{\phi\,\pi(\xx), (1 - \phi) \pi(\xx - \zz)\} 
\le 
2^{r - 2} \E \|\XX_0-\mm\|^r.
\end{equation}
Furthermore, for any  $\ccc\in\bbR^d$ and $m\in\bbR$ we have
\begin{equation}
\label{TVZc}
|\innprod{\ccc,\zz}|^r 
\sum_{\xx\in\bbZ^d}\!\min\{\phi\,\pi(\xx), (1 - \phi) \pi(\xx - \zz) \} \dd\xx 
\le 2^{r - 2} \E |\innprod{\ccc,\XX_0}-m|^r.
\end{equation}
\end{theorem}

\begin{proof}
We follow closely the proof of \refT{T:TVLB}, replacing integrals by sums.
We begin with \eqref{TVZc}. Again, this is trivial if $\innprod{\ccc,\zz}=0$,
and we may assume
$\innprod{\ccc,\zz}=1$.

Let $g(\xx):=\min\bigcpar{\phi\pi(\xx),(1-\phi)\pi(\xx-\zz)}$
and
$h(\xx):=|\innprod{\ccc,\xx}-m|^r\pi(\xx)$.
Let further $A:=\set{\xx\in\bbZ^d:\innprod{\ccc,\xx}\in[0,1)}$.
The translates
$A+n\zz=\set{\xx\in\bbZ^d:\innprod{\ccc,\xx}\in[n,n+1)}$
 form a partition of $\bbZ^d$, and thus for any
nonnegative function~$f$ we have
\begin{align}\label{jz1}
\sum_{\xx\in\bbZ^d}f(\xx)&
=\sum_{\xx\in A}\sum_n f(\xx+n\zz).
\end{align}

For every $\xx\in \bbZ^d$, we have
by \refL{L:TVLB equally spaced},
using \eqref{pineq2} with 
$b:=\innprod{\ccc,\xx}-m$
and
$p_n:=\pi(\xx+n\zz)$,
\begin{align}\label{jz2}
&  \sum_n g(\xx+n\zz)
=\sum_n \min\{(1-\phi)p_{n-1}, \phi p_n\}
\le 2^{r-2}\sum_n |b+n|^rp_n
\notag\\&\hskip4em
= 2^{r-2}\sum_n |\innprod{\ccc,\xx+n\zz}-m|^rp_n
=2^{r-2}\sum_n h(\xx+n\zz).
\end{align}
We sum this over $\xx\in A$, and, using~\eqref{jz1} twice, obtain
\begin{align}\label{jz3}
\sum_{\xx \in \bbZ^d} g(\xx)&
=\sum_{\xx\in A}\sumnz g(\xx+n\zz)
\le\sum_{\xx\in A} 2^{r-2}\sumnz h(\xx+n\zz)
\notag\\&
=2^{r-2}\sum_{\xx \in \bbZ^d}h(\xx)
=2^{r-2}\E|\innprod{\ccc,\XX_0}-m|^r
.\end{align}
This proves \eqref{TVZc}.

For \eqref{TVZ},
the case $\zz=\zzero$ is trivial, so we assume $\zz\neq\zzero$.
Let $\ee_\zz:=\zz/\norm{\zz}\in\bbR^d$, so that $\norm{\ee_\zz}=1$ and
$\innprod{\ee_\zz,\zz}=\norm{\zz}$. 
Then \eqref{TVZ} follows from \eqref{TVZc} with $\ccc:=\ee_\zz$ and
$m:=\innprod{\ccc,\mm}$. 
\end{proof}

\begin{remark}
\label{R:poss equal}
We can have equality in~\eqref{TVZ} and \eqref{TVZc}.
A simple example (for any $r$) with $d=1$
is provided by 
$\P(X_0 = \pm1)  = \frac12$, $z =2$, $\phi=\frac12$, $m=0$, and any $c\in\bbR$.
(See also the related examples discussed in \refE{E:discrete}.)
\end{remark}

\section{More on total variation distance} \label{A:more TV}

If we relax $<$ to $\leq$ in~\eqref{TVwithoutphi}, then the following main theorem of this appendix improves \refC{C:TVLB} in two different ways.

\begin{theorem}\label{T:TVLB3}
Let\/ $\XX_0$ be a random variable taking values in a 
real
Banach space~$E$,
and 
let\/ $r \in [1, \infty)$.
If\/ $\zz$ and\/ $\mm$  are any deterministic vectors in~$E$, then  
\begin{equation}
\label{TVwithoutphi3}
\|\zz\|^r [1 - \dtv(\XX_0, \XX_0 + \zz)] \le 2^{r - 1} \E \|\XX_0-\mm\|^r.
\end{equation}
Furthermore, for any continuous linear functional 
$\chi\in E^*$ and $m\in\bbR$ we have 
\begin{equation}
\label{TVCE}
|\chi(\zz)|^r [1 - \dtv(\XX_0, \XX_0 + \zz)] 
\le 2^{r - 1} \E |\chi(\XX_0)-m|^r.
\end{equation}
\end{theorem}

\begin{remark}
\label{R:TVLB3}
Note that, in this generality, equality is possible in \eqref{TVwithoutphi3}
and \eqref{TVCE} (for any $r$);
this is shown by the example in \refR{R:poss equal}.
\end{remark}

As shown in \refR{R:TVLB-}, the condition $r\ge2$ is necessary for
\eqref{pineq2} to hold in general, and in particular 
to  allow arbitrary $\phi\in\oi$. 
To  prove \refT{T:TVLB3}, we will first state and prove a lemma (\refL{L:TVLB2}) extending 
\refL{L:TVLB equally spaced} 
in the special case $\phi = \frac12$ to $r \geq 1$.

\begin{lemma}
\label{L:TVLB2}
Let $r \in [1, \infty)$.
Then, for any doubly infinite nonnegative sequence $(p_n)_{n \in \bbZ}$ 
and any $b \in \bbR$, we have \eqref{pineq2} with
$\phi =\frac12$, that  is,
\begin{equation}
\label{pineqC}
\frac12\sum_n \min\{p_{n - 1}, p_n\}
\le
2^{r - 2} \sum_n |b + n|^r p_n 
.\end{equation}
If further,
for simplicity, we assume that 
$b \in [0, 1)$, then equality 
holds in \eqref{pineq2}
if and only if 
either both sides are infinite or
one of the following (non-exclusive) sets of
conditions holds:
\begin{romenumerate}

\item \ref{Ltvlb0} in \refL{L:TVLB equally spaced};
\item \ref{Ltvlb1} in \refL{L:TVLB equally spaced};
\item \ref{Ltvlb2} in \refL{L:TVLB equally spaced};
\item \label{Ltvlb5}
\begin{alphiienumerate}
\item $p_n=0$  for $n \notin \{-1, 0\}$, 
\item $p_{-1} = p_0 $, and
\item $r=1$;
\end{alphiienumerate}
\item \label{Ltvlb4}
\begin{alphiienumerate}
\item $p_n = 0$ for $n \notin \{-1, 0, 1\}$,
\item $p_0 \geq \max\{p_{-1}, p_1\}$,
\item $r = 1$, and
\item $b = 0$.
\end{alphiienumerate}

\end{romenumerate}
\end{lemma}

\begin{remark}\label{R:TVLB+C}
As in \refL{L:TVLB equally spaced},
the conditions for equality for general $b$ are given by 
the conditions above applied to the fractional part $b-\floor{b}$ and the
shifted sequence $(p_{n-\floor{b}})_n$.
\end{remark}

\begin{remark}
The value $\phi = 1/2$ is truly special.  For any other given value of $\phi \in (0, 1)$ and for any given value of 
$r \in [1, 2)$, 
it is easy to verify that
we obtain a counterexample to~\eqref{pineq2} by
taking  $p_{-1}:=\phi$, $p_0:=1-\phi$, and $p_n:=0$ for $n\notin\set{-1,0}$
with $b$ as in~\eqref{sa2} (interpreted as $b=0$ or 1 when $r=1$).
[Note  that for $r\in(1,2)$, the inequality in \eqref{ponly} goes in the
opposite direction.]
\end{remark} 

\begin{proof}[Proof of \refL{L:TVLB2}]
  We follow the proof of \refL{L:TVLB equally spaced}.
Note first that, again, 
we may (by shifting $n$ by $\floor b$)
assume that $b\in[0,1)$;
we may also assume $\sum_n |b + n|^r p_n < \infty$.

Under the present assumption $r\ge1$, the \rhs{} of \eqref{sa1} is
non-negative at least when $|b+n|\ge2$.
Hence, by successively replacing $p_n$ by 0 for $n=2,3,-3,4,-4,\dots$
and using dominated convergence as before, we see that it suffices to prove
\eqref{pineqC} when $p_n=0$  for $n\notin\set{-2,-1,0,1}$.
Furthermore, assuming this, so in particular $p_2=0$, we see similarly to
\eqref{sa1} that if we replace $p_1$ by 0, then
the difference between the sides of~\eqref{pineqC} [as in \eqref{pineq1}]
decreases  by
\begin{align}\label{za1}
&2^{r - 2} |b + 1|^r p_1 - \tfrac12\min\{p_0, p_1\} 
-\tfrac12\min\{p_1, p_2\} 
\notag\\&\qquad
 \geq 2^{r - 2} (b + 1)^r p_1 - \tfrac12 p_1 
=\bigsqpar{2^{r - 2} (b + 1)^r - \tfrac12} p_1
\ge0
.\end{align}
Hence it suffices to prove \eqref{pineqC} when also $p_1=0$.
By the same argument, now using $p_{-3}=0$, we may also assume $p_{-2}=0$.
Hence we are reduced to the situation in \eqref{qp}, 
and it suffices to show that
\begin{align}
 2^{r - 2} [(1 - b)^r q + b^r p] - \tfrac12\min\{ q,  p\} 
\ge0
\label{qp2}   
\end{align}
for $b\in[0,1)$.
By scaling, we may again suppose that $q=1-p$.
Then the \lhs{} of \eqref{qp2} is a linear function of $p$ on the intervals
$[0,\frac12]$ and $[\frac12,1]$, and consequently it suffices to verify
\eqref{qp2} for the three cases 
(i) $p=0$ and  thus $q = 1$, 
(ii) $p=1$ and thus $q=0$, and
(iii) $p=q=\frac12$. 
The first two cases are trivial.
Finally, when $p=q=\frac12$, the minimum over $b\in[0,1)$ of the \lhs{} of
\eqref{qp2} is attained at $b=\frac12$. (In fact, for arbitrary $p,q>0$, 
the  minimizing value of~$b$ is given by \eqref{sa2} for every $r > 1$.) And in this case
$p=q=b=\frac12$, \eqref{qp2} is immediately verified (with equality).

We leave it to the reader 
to establish the cases of equality.
\end{proof}

\begin{remark}\label{R:TVLB2}
The
condition $r \geq 1$ is necessary for
\refT{T:TVLB3} and
\refL{L:TVLB2}
in general.
If $r < 1$, we may produce a counterexample to \refL{L:TVLB2} as follows.
Take $p_{-1} := \frac12$, $p_0 := \frac12$, 
$p_n:=0$ for $n\notin\set{-1,0}$,
and 
$b := 0$.
Then the \lhs{} of \eqref{pineqC} is
\begin{equation}
2^{r - 3} - \tfrac14 < 0.
\end{equation}
An (essentially equivalent) counterexample to \refT{T:TVLB3} 
with $E=\bbR$ 
is obtained by taking $X_0$ uniform on $\set{0,1}$, $z=1$, and $m=0$.
\end{remark}

\begin{proof}[Proof of \refT{T:TVLB3}]
We begin with \eqref{TVCE}.
The case $\chi(\zz)=0$ is trivial, so we assume 
$\chi(\zz)\neq0$, and by homogeneity, we may assume
$\chi(\zz)=1$.

The formula \eqref{dtv} for the total variation distance extends to
random variables with values in any measurable space
as
\begin{align}\label{dtv3z}
  \dtv(\YY,\tYY) 
= \frac12\int|\pif(\xx)-\tpif(\xx)|\dd\mu(\xx)
=1-\int \min\bigcpar{\pif(\xx),\tpif(\xx)}\dd\mu(\xx)
\end{align}
for any $\gs$-finite measure $\mu$ such that $\YY$ and $\tYY$ have
densities $f$ and $\tf$ with respect to $\mu$.

The proof is now similar to the proof of \refT{T:TVLB}, but since we 
in general
cannot use Lebesgue measure, we first construct a suitable dominating measure.
Let $\nu_n:=\cL(\XX_0+n\zz)$, the distribution of the translated random
vector $\XX_0+n\zz$, and define the infinite measure
\begin{align}\label{zb1}
  \mu:=\sum_{n\in\bbZ}\nu_n.
\end{align}
Then, 
for any measurable function $g\ge0$ on $E$,
\begin{align}\label{zc1}
  \int g(\xx)\dd\mu(\xx)
&=\sum_n  \int g(\xx)\dd\nu_n(\xx)
=\sum_n\E g(\XX_0+n\zz)
\notag\\&
=\E\sum_n g(\XX_0+n\zz)
.\end{align}

We first apply \eqref{zc1} when $g$ is the indicator function $\etta_{E_k}$
of the set $E_k:=\set{\xx\in E:\chi(\xx)\in[k,k+1)}$ for an integer $k\in\bbZ$.
We have, since $\chi(\zz)=1$,
\begin{align}
  \sum_n g(\xx+n\zz)&
=\sum_n \indic{\chi(\xx+n\zz)\in[k,k+1)}
\notag\\&
=\sum_n \indic{\chi(\xx)+n\in[k,k+1)}
=1
\end{align}
for every $\xx\in E$. Consequently, \eqref{zc1} yields 
$\mu(E_k)=\int \etta_{E_k}\dd\mu=1$.
In particular, since $E=\bigcup_k E_k$, this shows that $\mu$ is $\gs$-finite.

By construction, $\cL(\XX_0) = \nu_0 \ll \mu$ and $\cL(\XX_0+\zz)=\nu_1 \ll \mu$,
so $\XX_0$ and $\XX_0+\zz$ both have distributions that are absolutely
continuous with repect to $\mu$, and thus they have
densities $\pif$ and $\tpif$ with respect to $\mu$. It is easily seen that
we can take
$\tpif(\xx)=\pif(\xx-\zz)$. Hence, \eqref{dtv3z} applies and  yields
\begin{align}\label{dtv4z}
1- \dtv(\XX_0,\XX_0+\zz)
=\int \min\bigcpar{\pif(\xx),\pif(\xx-\zz)}\dd\mu(\xx).
\end{align}

Let  $g(\xx):=\min\bigcpar{\pif(\xx),\pif(\xx-\zz)}$
and
$h(\xx):=|\chi(\xx)-m|^r\pif(\xx)$.
For every $\xx\in E$, we have
by \refL{L:TVLB2},
using \eqref{pineqC} with 
$p_n:=\pif(\xx+n\zz)$ and
$b:=\chi(\xx)-m$, that
\begin{align}\label{zc2}
  \sum_n g(\xx+n\zz)&
=\sum_n \min\{p_n,p_{n-1}\}
\le 2^{r-1}\sum_n |b+n|^rp_n
\notag\\&
=2^{r-1}\sum_n h(\xx+n\zz)
.\end{align}
Consequently, by applying \eqref{zc1} first to $g$ and then to $h$, we obtain 
\begin{align}\label{zc3}
  \int &\min\bigcpar{\pif(\xx),\pif(\xx-\zz)}\dd\mu(\xx)
=  \int g(\xx)\dd\mu(\xx)
=\E\sum_n g(\XX_0+n\zz)
\notag\\&
\le 2^{r-1}\E\sum_n h(\XX_0+n\zz)
=2^{r-1}\int h(\xx)\dd\mu(\xx)
\notag\\&
=2^{r-1}\int|\chi(\xx)-m|^r\pif(\xx)\dd\mu(\xx)
=2^{r-1}\E |\chi(\XX_{0})-m|^r
,\end{align}
where the last equality holds  since~$\pif$ is the density of $\XX_0$.
The result \eqref{TVCE} follows from \eqref{dtv4z} and \eqref{zc3}.

To show \eqref{TVwithoutphi3}, take
a continuous linear functional $\chi$ on $E$ with $\chi(\zz)=\norm{\zz}$
and $\norm{\chi}=1$.
(This is possible by the Hahn-Banach theorem.)
Then \eqref{TVwithoutphi3} follows from \eqref{TVCE}, with $m:=\chi(\mm)$,
since
\begin{align}
  |\chi(\XX_0)-m|
=
  |\chi(\XX_0-\mm)|
\le \norm{\XX_0-\mm}.
\end{align}
\end{proof}

As noted above (\refR{R:TVLB3}), in general we can have equality in
\eqref{TVwithoutphi3}. However, if we again consider a random vector in
$\bbR^d$ having a density, then the strict inequalities 
in~\eqref{TVwithoutphi}--\eqref{TVCc} extend to $r>1$, and partly to $r\ge1$.

\begin{theorem}
\label{T:TVLB99}
Let\/ $\XX_0$ be a random vector in $\bbR^d$ having a density.
If\/
$\zz$ and\/ $\mm$  are any deterministic vectors in $\bbR^d$, 
and $r>1$,
then  
\begin{equation}
\label{TVwithoutphi99}
\|\zz\|^r [1 - \dtv(\XX_0, \XX_0 + \zz)] < 2^{r - 1} \E \|\XX_0-\mm\|^r.
\end{equation}
Furthermore, for any nonzero $\ccc\in\bbR^d$ and $m\in\bbR$ we have
\begin{equation}
\label{TVCE99}
|\innprod{\ccc,\zz}|^r [1 - \dtv(\XX_0, \XX_0 + \zz)] < 2^{r - 1} 
\E |\innprod{\ccc,\XX_0}-m|^r.
\end{equation}

If\/ $r=1$,
then \eqref{TVwithoutphi99} holds when $d\ge2$.
If\/ $r=1$ and $d=1$, 
then~\eqref{TVwithoutphi99} holds for 
every~$\zz$ except possibly two (opposite) values.
Similarly, if $r=1$, then
\eqref{TVCE99} holds (for a given $\ccc$)
for every vector $\zz$ except possibly two (opposite) values.
\end{theorem}

\begin{proof}
  We begin with \eqref{TVCE99}, and follow the proof of \eqref{TVc} in
  \refT{T:TVLB}.
Again, the result is trivial if $\innprod{\ccc,\zz}=0$, and we may assume
$\innprod{\ccc,\zz}=1$.
By replacing $\XX_0$ by $\XX_0-m\zz$, we may (for convenience)
also assume that $m=0$.

Suppose first $r>1$. If we have equality in \eqref{jr2}, then \refL{L:TVLB2}
yields the same conclusions as in the proof of \refT{T:TVLB}, and consequently
it is impossible to have equality in \eqref{TVCE99}; hence,
\eqref{TVCE99} holds.
This, as before, implies \eqref{TVwithoutphi99} by \eqref{a66}.
 
In the rest of the proof we consider the case $r=1$.
Then \refL{L:TVLB2}\ref{Ltvlb5}
shows another possibility for equality in
\eqref{jr2}: If $\xx$ belongs to the set~$A$ defined in \eqref{e99}, so that
$b=\innprod{\ccc,\xx}\in[0,1)$,
we may have
\begin{align}\label{b99}
  \begin{cases}
    \pi(\xx-\zz)=\pi(\xx)\ge0,
\\
\pi(\xx+n\zz)=0,\qquad n\notin\set{-1,0}.
  \end{cases}
\end{align}
The arguments in the proof of \refT{T:TVLB} show that in order to have
equality in \eqref{jr3} we have to have \eqref{b99} for a.e.\ $\xx\in A$.
[Note that \eqref{b99} holds trivially when $\sum_n\pi(\xx+n\zz)=0$.]
Consequently, $\pi(\xx)=0$ a.e.\ for $\xx\in\bigcup_{n\notin\set{-1,0}}(A+n\zz)$,
i.e., $\XX_0\in A\cup(A-\zz)$ \as,\ which means that \as
\begin{align}\label{aa99}
  \innprod{\ccc,\XX_0}>0\iff \XX_0\in A,
\qquad
  \innprod{\ccc,\XX_0}<0\iff \XX_0\in A-\zz.
\end{align}
It follows from \eqref{b99} also that 
$\P(\XX_0\in A)=\P(\XX_0\in A-\zz)=\frac12$, and that if we consider the
conditioned variables
\begin{align}
\XX_+&:= (\XX_0\mid \innprod{\ccc,\XX_0}>0), \label{xx+99}
\\
\XX_-&:= (\XX_0\mid \innprod{\ccc,\XX_0}<0),  \label{xx-99}
\end{align}
then we have the distributional equation
\begin{align}\label{xx99}
\XX_+\eqd \XX_-+\zz.
\end{align}
We have for convenience in the proof normalized $\ccc$ such that
$\innprod{\ccc,\zz}=1$, but \eqref{xx+99}--\eqref{xx-99}
are not affected by multiplying $\ccc$ by a positive constant, 
and replacing $\ccc$ by $-\ccc$ will just interchange $\XX_+$ and $\XX_-$.
Consequently, if we have equality in \eqref{TVCE99} for some non-zero
$\ccc$, 
then either \eqref{xx99} holds or
\begin{align}\label{xx99-}
\XX_+\eqd \XX_--\zz.
\end{align}
The distributional equations \eqref{xx99} and \eqref{xx99-} have at most one
solution~$\zz$ each (which are then the negatives of each other).
We have proved that \eqref{TVCE99} holds for all other $\zz\in\bbR^d$.

Turning to \eqref{TVwithoutphi99},
we first recall that the weak inequality \eqref{TVwithoutphi3} follows from
the weak version of \eqref{TVCE99} 
[i.e., \eqref{TVCE} with $\chi(\xx):=\innprod{\ccc,\xx}$] and \eqref{a66}.
However, in \eqref{a66} we have strict inequality unless $\XX_0-\mm$ is a
real multiple of $\ee_\zz$. If $d\ge2$, we thus \as{} have strict inequality
in \eqref{a66}, and consequently \eqref{TVwithoutphi99} holds.

If $d=1$, then \eqref{TVwithoutphi99} is equivalent to \eqref{TVCE99}.
\end{proof}

\begin{remark}
When $r=1$, it really is possible to have equality in \eqref{TVwithoutphi99} and  \eqref{TVCE99}.
One example is $d=1$, $X_0$ uniform
on $V:=(-1-a,-1+a)\cup(1-a,1+a)$ for some (any)
$a\in(0,1]$, $c=1$, $m = 0$, and $z=\pm2$.

In \eqref{TVCE99} we may have equality also for  $r = 1$ and $d\ge2$, possibly
for several $\cc$ (each with two $\zz$).
For example, let $\XX_0$ be uniform on $V^d$,  let $\cc$ be one of the
standard basis vectors $\ee_1,\dots,\ee_d$, let $\mm = 0$, and let $\zz=\pm2\cc$.
\end{remark}

\section{Symmetrization}
\label{A:symm}
In
\refS{S:unimodaltarget} we saw the usefulness of symmetrization of the target~$\pi$ for developing our results when the proposal step-density $\phi$ is symmetric.  The following general questions naturally arise:
\begin{enumerate}
\item[(i)] For fixed~$\phi$ (not necessarily symmetric), what is the effect of symmetrizing~$\pi$?
\item[(ii)] For fixed~$\pi$ (not necessarily symmetric), what is the effect of symmetrizing~$\phi$?
\end{enumerate}
In this appendix we address these two questions (in Sections~\ref{A:symmpi} and~\ref{A:symmphi}, respectively),
for convenience in the setting of integer-grid random walk Metropolis--Hastings (IRWMH); 
the results are the same in the
absolutely continuous RWMH setting, 
but it is simpler to give examples for IRWMH. 
(These examples can be carried over to RWMH by perturbation.) 

As we pointed out in \refS{SS:symmetrization}, symmetrization of~$\pi$ never has an effect on $\E \|\XX_0\|^r$, nor of course does symmetrization of~$\phi$.  So we may and do focus our attention on the effect of symmetrization on the incremental $r$th absolute moment $\E \|\XX_1 - \XX_0\|^r$, which, according to the IRWMH analogue of~\eqref{SRWMHr} is given by
\begin{equation}
\label{IRWMHr}
\E \|\XX_1 - \XX_0\|^r 
= \sum_{\zz \in \bbZ^d} \|\zz\|^r \sum_{\xx \in \bbZ^d} \min\{\pi(\xx) \phi(-\zz), \pi(\xx - \zz) \phi(\zz)\}.
\end{equation}
{\bf A useful class of one-dimensional examples for our treatment will take $X_0 \sim \pi$ to be distributed\ }Bernoulli$(p)${\bf\ and $Z \sim \phi$ with $\phi(1) = \gth = 1 - \phi(-1)$,\ }where $p \in (0, 1)$ and $\gth \in [0, 1]$.  For this class of examples we have from~\eqref{IRWMHr} that
\begin{align}
\E |X_1 - X_0|^r 
&= 2 \min\{(1 - p) \gth, p (1 - \gth)\}. \label{IRWMHr p gth}
\end{align}

\subsection{Symmetrization of~$\pi$}
\label{A:symmpi}
In \refS{SS:symmetrization} we saw that symmetrization of~$\pi$ for symmetric~$\phi$ can only increase 
$\E \|\XX_1 - \XX_0\|^r$.  But in this subsection we show, by using our class of one-dimensional examples, that symmetrization of~$\pi$ for \emph{general} fixed~$\phi$ multiplies 
$\E \|\XX_1 - \XX_0\|^r$ by a factor~$\ga$ that can be any number in $(1/2, \infty]$, and only such a number (assuming, to avoid trivialities, that $\E \|\bXX_1 - \bXX_0\|^r > 0$).

In general, using the discrete analogue
\begin{equation}
\label{Ibpi}
\bpi(\xx) := \tfrac12 [\pi(-\xx) + \pi(\xx)]
\end{equation}
of the notation~\eqref{bpi}, letting $\bXX = (\bXX_t)$ denote the corresponding IMH chain, and arguing much as we did in the proof of \refP{P:r bar versus not}, we find from~\eqref{IRWMHr} that
\begin{align}
\E \|\bXX_1 - \bXX_0\|^r 
&= \sum_{\zz \in \bbZ^d} \|\zz\|^r \sum_{\xx \in \bbZ^d} \min\{\bpi(\xx) \phi(-\zz), \bpi(\xx - \zz) \phi(\zz)\} \label{bar} \\
&\geq \tfrac12 
\left[ \sum_{\zz \in \bbZ^d} \|\zz\|^r \sum_{\xx \in \bbZ^d} \min\{\pi(\xx) \phi(-\zz), \pi(\xx - \zz) \phi(\zz)\} \right.
\nonumber \\
&{} \quad {} 
+ \left. \sum_{\zz \in \bbZ^d} \|\zz\|^r \sum_{\xx \in \bbZ^d} \min\{\pi(- \xx) \phi(-\zz), \pi(- \xx + \zz) \phi(\zz)\} \right]
\nonumber \\ 
&= \tfrac12 [\E \|\XX_1 - \XX_0\|^r + \E \|\XX_1(-) - \XX_0(-)\|^r] \nonumber \\
&\geq \tfrac12 \E \|\XX_1 - \XX_0\|^r, \label{half}
\end{align}
where $\XX(-) = (\XX(-)_t)$ denotes the MH chain corresponding to target $\xx \mapsto \pi(-\xx)$; thus 
$\ga \in [1/2, \infty]$.  It is not possible to have $\ga = 1/2$ because that would imply $\E \|\XX_1(-) - \XX_0(-)\|^r = 0$, from which it is not hard to show that $\E \|\XX_1 - \XX_0\|^r = 0$, too, which would in turn imply that 
$\E \|\bXX_1 - \bXX_0\|^r = 0$, contradicting an assumption we have made.

For our one-dimensional examples we have $\bpi(-1) = \bpi(1) = p / 2$ and $\bpi(0) = 1 - p$ and therefore, by~\eqref{bar},
\begin{align}
\lefteqn{\E |\bX_1 - \bX_0|^r} \nonumber \\ 
&= 2 \sum_{x \in \bbZ} \min\{\bpi(x) \gth, \bpi(x + 1) (1 - \gth)\} \nonumber \\
&= 2 \big[ \min\{\tfrac{p}{2} \gth, (1 - p) (1 - \gth)\} + \min\{(1 - p) \gth, \tfrac{p}{2} (1 - \gth)\} \big]. \label{bar1D}
\end{align}
Given $\ga \in (\tfrac12, \tfrac12 (\sqrt{2} + 1))]$, this last
expression~\eqref{bar1D} equals~$\ga$ times~\eqref{IRWMHr p gth} if $p =
\gth = \frac{2}{1 + 2 \ga} \in [2- \sqrt{2}, 1)$, and, 
given $\ga \in [\tfrac12 (\sqrt{2} + 1),\infty)$,
\eqref{bar1D} equals~$\ga$ times~\eqref{IRWMHr p gth} if 
$p = \frac{2}{1 + 2 \ga} \in (0, 2 -\sqrt{2}]$ 
and $\gth = \tfrac{2\ga-1}{2\ga} \in [2-\sqrt2,1)$.  
The possibility $\ga = \infty$ is covered by the following remark. 

\begin{remark}
If we assume that~$\phi$ is symmetric, then symmetrization of~$\pi$ multiplies $\E \|\XX_1 - \XX_0\|^r$ by a 
factor~$\ga$ that is at least~$1$ by \refP{P:r bar versus not} (and its IRWMH analogue) and can be any number in $[1, \infty]$.

To see this, suppose that $\ga \in [1, \infty]$ and
\begin{equation}
\label{pi5}
(\pi(-2), \pi(-1), \pi(0), \pi(1), \pi(2)) 
= \left( \frac{1}{4 \ga}, \frac{2 \ga - 1}{4 \ga}, 0, \frac{1}{4 \ga}, \frac{2 \ga - 1}{4 \ga} \right),
\end{equation}
with the interpretation that~\eqref{pi5} equals $(0, 1/2, 0, 0, 1/2)$ when $\ga = \infty$.  Suppose also that $\phi(1) = 1/2 = \phi(-1)$.  Then one computes $\E |X_1 - X_0|^r = \frac{1}{2 \ga}$ ($ = 0$ when $\ga = \infty$), finds that~$\bpi$ is uniform on 
$\{-2, -1, 1, 2\}$, and computes $\E |\bX_1 - \bX_0|^r = \frac12$.
\end{remark}

\subsection{Symmetrization of~$\phi$}
\label{A:symmphi}
Just as symmetrization of~$\pi$ for symmetric~$\phi$ can only increase $\E \|\XX_1 - \XX_0\|^r$ (\refP{P:r bar versus not}), so [starting from~\eqref{IRWMHr}] we find that symmetrization of~$\phi$ for symmetric~$\pi$ can only increase $\E \|\XX_1 - \XX_0\|^r$.  But in this subsection we show, again by using our class of one-dimensional examples, that symmetrization of~$\phi$ for \emph{general} fixed~$\pi$ (assumed nondegenerate) multiplies 
$\E \|\XX_1 - \XX_0\|^r$ by a factor~$\gb$ that can be any number in $(1/2, \infty]$, and only such a number (assuming, to avoid trivialities, that $\E \|\hXX_1 - \hXX_0\|^r > 0$).

Letting
\begin{equation}
\label{Ibphi}
\bphi(\zz) := \tfrac12 [\phi(-\zz) + \phi(\zz)],
\end{equation}
letting $\hXX = (\hXX_t)$ denote the corresponding IMH chain, and arguing as for~\eqref{half}, we find that
\begin{align}
\E \|\hXX_1 - \hXX_0\|^r 
&= \sum_{\zz \in \bbZ^d} \|\zz\|^r \bphi(\zz) \sum_{\xx \in \bbZ^d} \min\{\pi(\xx), \pi(\xx - \zz)\} \label{hat} \\
&\geq \tfrac12 \E \|\XX_1 - \XX_0\|^r, \label{hat half}
\end{align}  
and that $\gb = 1/2$ is ruled out in the same way that $\ga = 1/2$ was ruled out in \refS{A:symmpi}.

For our one-dimensional examples we have $\bphi(-1) = \bphi(1) = 1/2$ and therefore, by~\eqref{hat},
\begin{equation}
\label{hat1D}
\E |\hX_1 - \hX_0|^r = \sum_{x \in \bbZ^d} \min\{\pi(x), \pi(x - 1)\} = \min\{p, 1 - p\}.
\end{equation}
Given $\gb \in (1/2, 1]$, this last expression~\eqref{hat1D} equals~$\gb$ times~\eqref{IRWMHr p gth} if $p = \gth = \frac{1}{2 \gb} \in [\frac12, 1)$, and, given $\gb \in [1, \infty]$, \eqref{hat1D} equals $\gb$ times~\eqref{IRWMHr p gth} if $p = \frac12$ and $\gth = \frac{1}{2 \gb} \in [0, \frac12]$.

\begin{remark}
If we assume that~$\pi$ is symmetric, then symmetrization of~$\phi$ multiplies $\E \|\XX_1 - \XX_0\|^r$ by a 
factor~$\gb$ that is at least~$1$ (as discussed at the beginning of this subsection) and can be any number in 
$[1, \infty]$.

We can show this by a small modification of the case $\gb \in [1, \infty]$ discussed following~\eqref{hat1D}.  Indeed, in the present setting take $\pi(1) = \pi(-1) = 1/2$ and $\phi(2) = \gth = 1 - \phi(-2)$.  Then~\eqref{IRWMHr} yields 
$\E |X_1 - X_0|^r = 2^r \min\{\gth, 1- \gth\}$, and~\eqref{hat} yields $\E |\hX_1 - \hX_0|^r = 2^{r - 1}$.  The ratio
 $\E |\hX_1 - \hX_0|^r / \E |X_1 - X_0|^r$ equals $\frac{1}{2 \min\{\gth, 1 - \gth\}}$, which, by  taking $\gth = \frac{1}{2 \gb} \in [0, 1/2]$, can be any number $\gb \in [1, \infty]$. 
\end{remark}

\section{Symmetric discrete unimodal target and odd proposal steps}
\label{A:IRWMHunimodalr}

This appendix
is devoted to a study of the supremum $s(r)$ of the ratio $\E |X_1 - X_0|^r
/ \E |X_0|^r$ taken over
nondegenerate symmetric discrete unimodal distributions for $X_0 \in L^r$ and
proposal step-distributions concentrated on the odd integers. (The case $r
= 2$ was handled in \refT{T:Isymmunimodal}.)
For $r \in (0, \infty)$ and $k \in \bbN$, let
\begin{equation}
S(k; r) := \sum_{i = 1}^k i^r.
\end{equation}

\begin{theorem}
\label{T:Isymmunimodalr}
In the IRWMH setting, let $r \in (0, \infty)$, and consider
the supremum $s(r)$ of $\E |X_1 - X_0|^r / \E |X_0|^r$ taken over
nondegenerate symmetric  discrete unimodal
distributions~$\pi$ for $X_0 \in L^r$ and symmetric proposal step-distributions
concentrated on the odd integers. Then
$s(r) \geq 1$, $s(r) = 1$ if $r \in (0, 1]$, and if $r \notin (1, 1.043)$ we
 have
\begin{equation}
\label{sr simpler}
s(r) 
= \max_{k \in \bbN:\,k \leq k_0(r)} g(k; r),
\end{equation}
where
\begin{equation}
\label{gkr}
g(k; r) := \frac{(2 k - 1)^r}{S(k; r)}
\end{equation}
and
\begin{equation}
\label{k0r}
k_0(r) := \floor{\tfrac12 + (1 - 2^{-1/r})^{-1}} \geq 1.
\end{equation}
Further, $s(r)$ is strictly
increasing
on $[1.043, \infty)$.
\end{theorem}

We generalize \refT{T:Isymmunimodalr} (somewhat) to higher dimensions in \refT{T:Isymmunimodalr high}.

\begin{remark}
\label{R:Isymmunimodalr}
(a)~A special case of \refT{T:Isymmunimodalr} is the inequality~\eqref{Icorr1/10symm} in \refT{T:Isymmunimodal}.  Indeed, by~\eqref{k0r} we have $k_0(2) = 3$, and therefore by \eqref{sr simpler}--\eqref{gkr} we have
\begin{equation}
s(2) = \max\{1, \tfrac95, \tfrac{25}{14}\} = \tfrac95.
\end{equation}
This is easily seen to be equivalent to~\eqref{Icorr1/10symm}.

(b) Equation~\eqref{sr simpler} allows us to compute $s(r)$ exactly for any particular value  of 
$r \in [1.043, \infty)$.  
It also provides explicit formulas for $s(r)$ for various $r$-intervals.  For example, using~\eqref{sr simpler} and~\eqref{k0r} one finds 
for $r \in [16 / 5, 21 / 5] = [3.2, 4.2]$ that
\begin{equation}
s(r) = g(4; r) = \frac{7^r}{4^r + 3^r + 2^r + 1}.
\end{equation}

(c) Our
proof (later in this appendix) of~\eqref{sr simpler} for $r \in [1.043, \infty)$ is slightly incomplete, since (i)~for 
$r \in [2.198, \infty)$ we rely on asymptotics which imply that~\eqref{sr simpler} holds for all $r \geq r_0$ for some $r_0$ and the clear indication from a plot that a certain piecewise smooth function is nonnegative for $r \in [2.198, \infty)$; and (ii)~similarly for $r \in [1.043, 2.198)$ we rely on the clear indication from a plot that a certain function is strictly positive.  We are, however, quite confident that with enough tireless effort one could determine a sufficient value of $r_0$ and, by computing enough function values and derivative bounds to establish what is clearly indicated by the plots, give a rigorously complete proof.

Our proof that~\eqref{sr simpler} holds for all sufficiently large~$r$ \emph{is} rigorously complete, as is our proof of \refT{T:Isymmunimodalr+} except for part~(b), and our proof of that part is rigorously complete for all sufficiently large~$r$.

(d)~We  will discuss computation for $r \in (1, 1.043)$ in \refS{A:compute}.  Based on that discussion, we strongly believe the following conjecture.
\end{remark}

\begin{conj}
\label{Conj:s g}
For all $r \in (0, \infty)$ we  have
\begin{equation}
\label{s g conj}
s(r) 
= \sup_{k \in \bbN} g(k; r) 
= \max_{k \in \bbN:\,k < (r + 1) 2^r} g(k; r)
= \max_{k \in \bbN:\,k \leq \ceil{r}} g(k; r).
\end{equation}
\end{conj}
It is only the first and third equalities in~\eqref{s g conj} that are conjectures.  The second equality is clear, because, by the integral  comparison $S(k; r) \geq k^{r + 1} / (r + 1)$, we have
\begin{equation}
g(k; r) \leq (r + 1) 2^r (1 - \tfrac{1}{2 k})^r k^{-1} \leq (r + 1) 2^r k^{-1} \leq 1 = g(1; r), 
\end{equation}
with the last inequality holding if $k \geq (r + 1) 2^r$.
\smallskip

We also believe the following greater detail:

\begin{conj}
\label{Conj:Kr}
For each $k \geq 1$ there is a unique real root $r = r_k$ to the equation $g(k; r) = g(k + 1; r)$ (with $r_1 = 1$), and $r_k \in [1, \infty)$ increases strictly monotonically to $\infty$ as $k \to \infty$; indeed 
$r_k \sim k \ln 3$ as $k \to \infty$.  If $r \in (r_{k - 1}, r_k)$ with $k \geq 2$, then $s(r) = g(j; r)$ uniquely for $j = k$.  If $r = r_k$, then $s(r) = g(j; r)$ precisely for $j \in \{k, k + 1\}$.
\end{conj}

Along the lines of \refConj{Conj:Kr}, we will discuss pinning down the $\arg \max$ in $s(r) = \max_{k \in \bbN} g(k; r)$ for $r \in [1.043, \infty)$ in \refS{A:argmax}.  The main result there is \refT{T:Isymmunimodalr+}, which improves on \refT{T:Isymmunimodalr}.
\medskip

We prepare for the proof of \refT{T:Isymmunimodalr} with five helpful lemmas (Lemmas~\ref{L:helpful1}, \ref{L:Jensen+}, \ref{L:helpful2}, \ref{L:helpful3}, and~\ref{L:helpful4}).  The proof of \refT{T:Isymmunimodalr} can be found after the proof of~\refL{L:helpful4}.

\begin{lemma}
\label{L:helpful1}
For $r \in (0, \infty)$ we  have
\begin{equation}
\label{ISRWMHunimodalr}
s(r) = \sup_{k \in \bbN}\,\max_{y \in \bbZ:\,0 \leq y \leq k - 1} f(k, y; r),
\end{equation}
where 
\begin{equation}
\label{fkyr}
f(k, y; r) := \frac{(2 y + 1)^r (k - y)}{S(k; r)} \leq f_1(k; r)
\end{equation}
with
\begin{equation}
\label{f1kr}
f_1(k; r) := \frac{1}{2} \frac{r^r}{(r + 1)^{r + 1}} \frac{(2 k + 1)^{r + 1}}{S(k; r)}.
\end{equation}
\nopf
\end{lemma}

\begin{proof}
The proof of~\eqref{ISRWMHunimodalr}, with $f(k, y; r)$ defined at~\eqref{fkyr}, follows the same argument as for $r = 2$ in \refT{T:Isymmunimodal} and so is left to the reader.  The expression $f(k, y; r)$ is log-concave in real~$y$ and thus is maximized over real~$y$ when the derivative of $\ln f(k, y; r)$ vanishes, namely, at 
\begin{equation}
\label{real maximizing y}
y = \frac{2 r k - 1}{2 (r + 1)},
\end{equation}
with value $f_1(k; r)$ given by~\eqref{f1kr}.
\end{proof}

\begin{remark}
\label{R:bounds comparison}
(a)~For every $r \in (0, \infty)$ we have
\begin{equation}
\left( \frac{2 r}{r + 1} \right)^r \leq s(r) \leq 2^{r - 1}.
\end{equation}
The upper bound follows from the more general \refT{T:Imain}, and the lower bound (which we note is the constant appearing in \refT{T:unimodal-d}) follows by applying the inequality $s(r) \geq f(k, y_k; r)$, valid for any sequence $(y_k)$ of integers satisfying $0 \leq y \leq k - 1$, to any sequence for which $\frac{y_k}{k} \to \frac{r}{r + 1}$ (the maximizing ``$v$'' in the proof of \refL{L:Winkler}) and passing to the limit.

As we shall see in \refC{C:sasy}, $s(r) \sim \frac{2 \sqrt{3}}{9} 2^r$ as $r
\to \infty$.  Thus for symmetric unimodal (absolutely continuous) targets,
symmetric discrete unimodal targets with  symmetric odd proposal
steps, and general targets (in either the absolutely continuous
setting or the discrete setting), 
we have $\sup (\E |X_1 - X_0 | / \E |X_0|^r) \sim c 2^r$ as $r \to \infty$,
with~$c$ equal to $1 / e \approx 0.368$,
$2 \sqrt{3} / 9 \approx 0.385$, and $1/2$, respectively.  
(See \refT{T:unimodal} with \refR{R:sharp}, \refT{T:Isymmunimodalr}, 
and \refT{T:Imain} with \refE{E:discrete}.)
Thus we might regard the shrinkage factor $2 / e \approx 0.74$ as the gain from restricting to symmetric unimodal targets, the factor $4 \sqrt{3} / 9 \approx 0.77$ as the gain from restricting to unimodal targets to the situation in \refT{T:Isymmunimodalr}, and the modest increase factor $2 e \sqrt{3} / 9 \approx 1.05$ as ``the cost of discreteness'' for symmetric unimodal targets.

(b)~For $r \in (0, 1]$, the symmetric unimodal constants $\left( \frac{2 r}{r + 1} \right)^r$ are better than the constant~$1$ in \refT{T:Isymmunimodalr} by a shrinkage factor that is log-convex, equaling~$1$ at $r = 1$ and in the limit as $r \to 0$ and never any smaller than approximate value $0.79$ at $r \approx 0.30$.       
\end{remark}

\begin{lemma}\label{L:Jensen+}
Let $b > 0$, and suppose that $\psi:(-b, b) \to \bbR$ is convex.
Then the 
average
\begin{align}
\label{Psitau}
\Psi(\tau) 
:= 
\frac12 \int_{-1}^1 \psi(x \tau) \dd x
\end{align}
is a nondecreasing function of $\tau \in [0, b)$.
\end{lemma}

Stronger results are given in \cite[Theorem 5]{BrucknerOstrow(1962)}; 
we give a simple direct proof for completeness.

\begin{proof}
If $0 \le \tau_1 < \tau_2 < b$, then, by the convexity of~$\psi$, for any $x \in (0, 1)$ we have
\begin{align}
\psi(x \tau_1) + \psi(-x \tau_1) \le \psi(x \tau_2)+\psi(-x \tau_2). 
\end{align}
Now integrate over $x \in (0, 1)$.
\end{proof}

\begin{lemma}
\label{L:helpful2}
Suppose 
that $r \in [1, \infty)$.
\smallskip

\noindent
{\rm (a)}~For real $a, t \in [0, \infty)$ we have
\begin{align}
h(t, a; r) 
&:= [1 + (3 + a) t]^r \Bigl( [1 + (2 + a) t]^{r + 1} - [1 + a t]^{r + 1} \Bigr) \nonumber \\
&{} \qquad {} - [1 + (1 + a) t]^r \Bigl( [1 + (4 + a) t]^{r + 1} - [1 + (2 + a) t]^{r + 1} \Bigr) \geq 0. \label{htar ineq}
\end{align}

\noindent
{\rm (b)}~The function $f_1(k; r)$ defined at~\eqref{f1kr} is nonincreasing in integer $k \geq 1$.
\end{lemma}

In passing, we conjecture a considerable strengthening of~\eqref{htar ineq} when  $r \in \bbN$:

\begin{conj}
\label{Conj:strong!}
If  $r \in \bbN$ is arbitrary but fixed, then the polynomial $h(t, a; r)$ in $(t, a)$ has nonnegative coefficients. 
\end{conj}

\begin{proof}[Proof of \refL{L:helpful2}]
(a)~We will prove that, for fixed $r \in [1, \infty)$ and $t \in [0, \infty)$ the  function
\begin{equation}\label{h1}
h_1(t, a; r) := [1 + (1 + a) t]^{-r} \Bigl( [1 + (2+ a) t]^{r + 1} - [1 + a t]^{r + 1} \Bigr) 
\end{equation}
is nonincreasing in $a \in [0, \infty)$, and therefore
\begin{equation}\label{h1z}
h_1(t, a + 2; r) \leq h_1(t, a; r),
\end{equation}
which is equivalent to~\eqref{htar ineq}.

With $h_1(t,a;r)$ as  in~\eqref{h1}, and assuming $t\ge0$, we have
\begin{align}\label{h1a}
  h_1(t,a,r)&=[1+(1+a)t]^{-r}(r+1)t\int_{-1}^1\bigsqpar{1+(1+a+x)t}^r\dd x
\notag\\&
=(r+1)t\int_{-1}^1\Bigsqpar{1+\frac{xt}{1+(1+a)t}}^r\dd x
\notag\\&
=(r+1)t\int_{-1}^1\bigpar{1+x\tau}^r\dd x
,\end{align}
where
\begin{align}
  \tau=\tau(t,a):=\frac{t}{1+(1+a)t}.
\end{align}
Since $\tau$ is a nonincreasing function of $a$,
it suffices to show that the final integral in \eqref{h1a}
is nondecreasing in $\tau\in[0,1)$, which is the special  case $\psi(z) := (1 + z)^r$ and $b = 1$ 
of  \refL{L:Jensen+}.

(b)~For this it suffices to show for $k \geq 0$  that
\begin{equation}
\frac{f_1(k + 1; r)}{f_1(k; r)} = \left( \frac{2 k + 3}{2 k + 1} \right)^{r + 1} \frac{S(k; r)}{S(k + 1; r)} \leq 1,
\end{equation}
or equivalently  that
\begin{equation}
\label{Hkrineq}
S(k; r) \leq (k + 1)^r \left[ \left( \frac{2 k + 3}{2 k + 1} \right)^{r + 1} -  1 \right]^{-1}.
\end{equation}
Since one checks that~\eqref{Hkrineq} holds for all $r \in [1, \infty)$ when $k = 1$, it suffices to show that for all $k \geq 1$ we  have
\begin{align}
\lefteqn{S(k + 1; r) - S(k; r) = (k + 1)^r} \nonumber \\
&\leq (k + 2)^r \left[ \left( \frac{2 k + 5}{2 k + 3} \right)^{r + 1} -  1 \right]^{-1}
- (k + 1)^r \left[ \left( \frac{2 k + 3}{2 k + 1} \right)^{r + 1} -  1 \right]^{-1}, 
\end{align}
which, omitting some calculations, is equivalent to
\begin{equation}
h((2 k + 1)^{-1}, 0; r) \geq 0.
\end{equation}
But this follows from~\eqref{htar ineq}.
\end{proof}

\begin{lemma}
\label{L:helpful3}
Let $r \in (0, \infty)$.  Given  $k \in \bbN$, the expression $f(k, y; r)$
is maximized over integers $y \in [0, k - 1]$ by $y = k - 1$ if and only if $k \leq k_0(r)$.
\end{lemma}

\begin{proof}
Both conditions in the ``if and only if'' claim are satisfied by $k = 1$, since then 
$\frac12 + (1 - 2^{-1/r})^{-1} > 3/2 > 1 = k$.  So we may assume $k \geq 2$.
The expression $f(k, y; r)$ is log-concave in real~$y$ and so is maximized over integers $y \in [0, k - 1]$ by $y = k - 1$ if and only if its value at $y = k - 1$ dominates the value at $y = k - 2$.  This is easily seen to be equivalent to $k \leq \tfrac12 + (1 - 2^{-1/r})^{-1}$, which is equivalent to $k \leq k_0(r)$ because~$k$ is an integer.
\end{proof}

Because of its central appearance in~\eqref{gkr} and \eqref{fkyr}--\eqref{f1kr}, we next consider asymptotics  for 
$S(k; r)$.  As it turns out, we will need two-term asymptotics when $r \to \infty$ with~$k$ growing approximately linearly in~$r$.

\begin{lemma}
\label{L:helpful4}
Consider
any $c \in (0, \infty)$ and any $k \equiv k(r)$ with $k(r) \sim c^{-1} r$ as $r \to \infty$.  Let 
$x \equiv x(c) := e^{-c} \in (0, 1)$ and $\gamma \equiv \gamma(r) := r / k$ (which converges to~$c$) and $\xi \equiv \xi(r) := e^{- \gamma}$ (which converges to~$x$).  Then as $r \to \infty$ we  have
\begin{equation}
\label{Hasy}
S(k; r) = k^r (1 - \xi)^{-1} \Bigl[ 1 - (1 + o(1)) \tfrac12 c^2 x (1 + x) (1 - x)^{-2} r^{-1} \Bigr]. 
\end{equation} 
\end{lemma}

\begin{proof}
We give only a broad (but illuminating) sketch of the proof.  The idea is to  write
\begin{equation}
S(k; r) = k^r \sum_{j = 0}^{k - 1} \left( 1 - \frac{j}{k} \right)^r
\end{equation}
and when $j = o(r\qq)$ use 
\begin{equation}
\ln \left[ \left( 1 - \frac{j}{k} \right)^r \right] = r \ln \left( 1 - \frac{j}{k} \right) 
= - j \gamma - (1 + o(1)) \frac12 j^2 \gamma^2 r^{-1},
\end{equation}
so that
\begin{align}\label{dx1}
\left( 1 - \frac{j}{k} \right)^r 
&= e^{- j \gamma} \bigl[ 1 - (1 + o(1)) \tfrac12 j^2 \gamma^2 r^{-1} \bigr] \nonumber \\
&= \xi^j \bigl[ 1 - (1 + o(1)) \tfrac12 j^2 c^2 r^{-1} \bigr].
\end{align}
Finally, $\sum_{j = 0}^{\infty} \xi^j = (1 - \xi)^{-1}$ and
\begin{equation}
\sum_{j = 0}^{\infty} \xi^j j^2 
= \xi (1 + \xi) (1 - \xi)^{-3}
\sim x (1 + x) (1 - x)^{-2} (1 - \xi)^{-1}.
\end{equation}
\end{proof}
\vspace{-.01in}

\begin{proof}[Proof of \refT{T:Isymmunimodalr}]
If $k \geq 1$, then $2 k - 1 \geq k$ and so
\begin{equation}
g(k; r) = \left[ \sum_{i = 1}^k \left( \frac{i}{2 k - 1} \right)^r \right]^{-1} 
\end{equation}
is nondecreasing in $r \in (0, \infty)$, and strictly increasing if $k \geq
2$.  
The final sentence in the statement of the theorem then follows provided
that we verify the other assertions,
noting that for $r>1$ we have $k_0(r)\ge2$ and $g(2,r)>1=g(1,r)$.

We prove (a) $s(r) \geq 1$ for $r \in (0, \infty)$, (b) $s(1) = 1$, (c) $s(r) = 1$ for $r \in (0, 1)$, and 
equation~\eqref{sr simpler} for (d) $r \in [2.198, \infty)$ and (e) $r \in [1.043, 2.198)$.
\smallskip

(a)~It is obvious from~\eqref{ISRWMHunimodalr} that $s(r) \geq 1$, since $f(1, 0; r) = 1$.
\smallskip

(b)~When $r = 1$, \eqref{fkyr} simplifies to
\begin{equation}
\label{f1ky}
f(k, y; 1) = \frac{2 (2 y + 1) (k - y)}{k (k + 1)},
\end{equation}
maximized over integer
$y \in [0, k]$ uniquely by $y = \floor{k/2}$, with
\begin{equation}
\label{f1kk2}
f(k, \floor{k/2}; 1) = 1.
\end{equation}
Thus $s(1) = 1$.
\smallskip

(c)~We will show that $f(k, y; r) \leq 1$ for all integers $0 \leq y \leq k - 1$, and then the result follows because [recall~(a) in this proof] $s(r) \geq 1$, too.  For this, fix $r \in (0, 1)$ and integer $y \geq 0$ and  consider
\begin{align}
b(k) 
&\equiv b(k; r, y) \nonumber \\ 
&:= S(k; r) [1 - f(k, y; r)] = S(k; r) - (2 y + 1)^r (k - y), \quad k \geq y. \label{bky}
\end{align}
Observe  that $b(y) = S(y; r) > 0$ and that for $k \geq y$ we have
\begin{equation}
b(k + 1) - b(k) 
= (k + 1)^r - (2 y + 1)^r
\begin{cases}
< 0, & \mbox{if $k < 2 y$;} \\
= 0, & \mbox{if $k = 2 y$;} \\
> 0, & \mbox{if $k > 2 y$.}
\end{cases}
\end{equation}
Thus  
\begin{align}
b(k) 
&\geq b(2 y) = b(2 y + 1) = S(2 y; r) - y (2 y + 1)^r \nonumber \\
&= (2 y + 1)^r \left[ \sum_{i = 1}^{2 y} \left( \frac{i}{2 y + 1} \right)^r - y \right]
=: (2 y + 1)^r\,\tb(r; y).
\end{align}
Clearly, $\tb(r; y)$ is decreasing in~$r$, so 
\begin{equation}
\tb(r; y) \geq \lim_{\tr \uparrow 1} \tb(\tr; y) = 0.
\end{equation}

(d)~Let $r \in (2.198, \infty)$.  To prove~\eqref{sr simpler}, we will show that 
\begin{equation}
\label{f1kr bound}
\mbox{$f_1(k; r) \leq g(\ceil{r / \ln 3}; r)$ for all $k \geq k_0(r) + 1$}.  
\end{equation}
It then follows from~\eqref{ISRWMHunimodalr} and \refL{L:helpful3}, using~\eqref{gkr}, that 
\begin{align}
s(r)
&= \max\left\{ \max_{k \in \bbN:\,k \leq k_0(r)}\,\max_{y \in \bbZ:\,0 \leq y \leq k - 1} f(k, y; r), \right. \nonumber \\
&{} \qquad \qquad \left. \sup_{k \in \bbN:\,k \geq k_0(r) + 1}\,\max_{y \in \bbZ:\,0 \leq y \leq k - 1} f(k, y; r) \right\} \nonumber \\
&= \max\left\{ \max_{k \in \bbN:\,k \leq k_0(r)} g(k; r),
\sup_{k \in \bbN:\,k \geq k_0(r) + 1}\,\max_{y \in \bbZ:\,0 \leq y \leq k - 1} f(k, y; r) \right\}
\label{max max max sup max}
\end{align}
and by~\eqref{fkyr} and~\eqref{f1kr bound} that
\begin{align}
\lefteqn{\sup_{k \in \bbN:\,k \geq k_0(r) + 1}\,\max_{y \in \bbZ:\,0 \leq y \leq k - 1} f(k, y; r)} \nonumber \\
&\leq \sup_{k \in \bbN:\,k \geq k_0(r) + 1} f_1(k; r) 
\leq g(\ceil{r / \ln 3}; r) 
\leq \max_{k \in \bbN:\,k \leq k_0(r)} g(k; r), \label{sup max}  
\end{align}
where the last inequality follows since it is not hard to verify that $\ceil{r / \ln 3} \leq k_0(r)$ for all 
$r \in (0, \infty)$.  Combining \eqref{max max max sup max}--\eqref{sup max} then yields~\eqref{sr simpler}.

We now prove~\eqref{f1kr bound} by invoking \refL{L:helpful2}(b) and showing that
\begin{equation}
\label{f1ineq}
f_1(k_0(r) + 1; r) \leq g(\ceil{r / \ln 3}; r).
\end{equation}

A plot of the ratio of the two sides of~\eqref{f1ineq} clearly indicates that the inequality holds for 
$r \in [2.198, \infty)$.  To give a formal proof,
we note first that as $r\to\infty$, \eqref{k0r} 
yields $k_0(r)\sim r/\ln 2$,
and then \refL{L:helpful4} and the definitions~\eqref{f1kr} and~\eqref{gkr}
yield
\begin{align}
  f_1(k_0(r)+1; r) &\sim \frac{1}{e\sqrt2 \ln 2}2^r,
\\
g(\ceil{r/\ln 3}; r)&\sim \frac{2}{3\sqrt 3} 2^r.
\end{align}
Numerically, the two constants  are, respectively, 
\begin{align}
\frac{1}{e\sqrt2 \ln 2}\approx 0.375,
\qquad
\frac{2}{3\sqrt 3} \approx 0.385,
\end{align}
and thus \eqref{f1ineq} holds for $r\ge r_0$ for some $r_0$.
One then can
use numerical analysis 
[recall \refR{R:Isymmunimodalr}(c)]
to confirm the inequality for $r \in [2.198, r_0)$
(analysis we omit).

(e)~Let 
$r \in [1.043, 2.198)$.  By \refL{L:helpful1} and the same sort of argument as for part~(d), we need only 
note [by two single plots of the difference, or ratio, of the two sides of~\eqref{f16r} and
of~\eqref{maxes}, plots that can be made 
into a complete rigorous proof
using numerical analysis] that 
\begin{equation}
\label{f16r}
f_1(6; r) < \max\{g(2; r), g(3; r)\}
\end{equation}
and
\begin{equation}
\label{maxes}
\max_{(k, y) \in \bbZ^2 \setminus \{(2, 1), (3, 2)\}:\,0 \leq y \leq k - 1 \leq 4} f(k, y; r) < \max\{g(2; r), g(3; r)\},
\end{equation} 
where furthermore the \rhs{s} equal $g(2,r)$ if $k_0(r)=2$, i.e., if
$r<\ln2/\ln(5/3)\approx 1.357$.
\end{proof}

\refT{T:Isymmunimodalr} allows us to compute asymptotics of $s(r)$ as $r \to \infty$ and, as the proof of the following corollary to \refT{T:Isymmunimodalr} will demonstrate, any value~$k$ achieving $s(r)$ [as in~\eqref{sr simpler}].

\begin{corollary}
\label{C:sasy}
As $r \to \infty$ we have
\begin{equation}
s(r) \sim \frac{2 \sqrt{3}}{9} 2^r.
\end{equation}
\end{corollary}

\begin{proof}
Because~\eqref{sr simpler} holds for all large~$r$, with $k_0(r)$ defined at~\eqref{k0r} and satisfying 
$k_0(r) \sim r/\ln 2 = O(r)$, it suffices to show that if $k \equiv k(r)$
satisfies $k \sim c^{-1} r$ as $r \to \infty$, then, with the notation $x = e^{-c}$ introduced in the statement of \refL{L:helpful4},
\begin{equation}
\label{sim}
g(k, r) \sim x^{1/2} (1 - x) 2^r 
\end{equation}
and that if $c \geq 2 \ln(3 \sqrt{3} / 2)$ (\ie,\ $x \leq 4/27$) then
\begin{equation}
\label{max bound}
\max_{j \leq k} g(j, r) \leq \frac{2 \sqrt{3}}{9} 2^r.
\end{equation}
Then, because $x^{1/2} (1 - x)$ is maximized over $x \in (0, 1)$ by $x = 1 / 3$ and 
$(1/3)^{1/2} [1 - (1/3)] = 2 \sqrt{3} / 9$, the corollary follows. 

Given our preceding work, the proof of~\eqref{sim} is quite simple.  Indeed, we have 
$(2 k - 1)^r \sim 2^r k^r x^{1/2}$ and, by use of the first of the two terms in~\eqref{Hasy}  in \refL{L:helpful4},
$S(k; r) \sim k^r (1 - x)^{-1}$.  Thus, by the definition~\eqref{gkr}, \eqref{sim} holds.

Proving~\eqref{max bound} is even much simpler.  We apply the trivial  inequality $S(j; r) \geq j^r$, so that 
for $j \leq k \sim c^{-1} r$ we  have
\begin{align}
g(j; r) 
&= \frac{(2 j - 1)^r}{S(j; r)} 
\leq 2^r \left( 1 - \frac{1}{2 j} \right)^r \leq 2^r e^{- r/ (2 j)} \nonumber \\
&\leq 2^r e^{- r / (2 k)} 
= (1 + o(1)) e^{-c/2} 2^r
\leq (1 + o(1)) \frac{2 \sqrt{3}}{9} 2^r.   
\end{align}
\end{proof}

We next note that \refT{T:Isymmunimodalr} is extended easily to higher dimensions, according to the following IRWMH analogue of \refT{T:unimodal-d}, if we place a suitable restriction on the support of the proposal 
step-distribution. 

\begin{theorem}
\label{T:Isymmunimodalr high}
In the IRWMH setting in dimension~$d$, let $r \in (0, \infty)$, and consider
the (one-dimensional) supremum $s(r)$ as defined in the statement of \refT{T:Isymmunimodalr}.  If
\begin{equation}
\label{Istar}
\pi(\xx) = \hpi(\|\xx\|)\mbox{\rm \ for some nonincreasing function~$\hpi$ on $[0, \infty)$}
\end{equation}
and the proposal step-distribution is symmetric and concentrated on odd multiples of the coordinate vectors in 
$\bbZ^d$, then
\begin{equation}
\label{sr ineq}
\E \|\XX_1 - \XX_0\|^r \leq s(r) \E \|\XX_0\|^r.
\end{equation}
\end{theorem}

\begin{proof}
Denote the probability mass function for the proposal step-distribution by~$\phi$.
Proceeding analogously in the IRWMH setting to how we proceeded in the proof
of the RWMH \refT{T:unimodal-d}, we note that
$\P\!\left( |\langle \XX_0, \zz \rangle| = \tfrac12 \|\zz\|^2 \right) = 0$ 
when $\phi(\zz) > 0$ 
by our hypothesis about~$\phi$
and therefore have 
\begin{equation}
\label{high1}
\E \|\XX_1 - \XX_0\|^r 
= \sum_{\zz \in \bbZ^d} \phi(\zz) \|\zz\|^r \P\!\left( |\langle \XX_0, \zz \rangle| > \tfrac12 \|\zz\|^2 \right). 
\end{equation}
By assumption, if $\phi(\zz) > 0$, then $\zz = (2 y + 1)\,\ee_i$ for some $i \in \{1, \ldots, d\}$ and $y \in \bbZ$ and so
\begin{equation}
\label{high2}
\|\zz\|^r \P\!\left( |\langle \XX_0, \zz \rangle| > \tfrac12 \|\zz\|^2 \right)
= |2 y + 1|^r \P\!\left( |X_{0 i}| > \tfrac12 |2 y + 1| \right).
\end{equation}
But it follows from~\eqref{Istar} that each $X_{0 i}$ has a distribution
that is discrete unimodal (about~$0$), 
so by the (one-dimensional) definition of $s(r)$ 
in \refT{T:Isymmunimodalr}, 
see also \eqref{ISRWMHunimodalr} and \eqref{rsymmuni discrete},
we have 
\begin{align}
\lefteqn{\hspace{-.7in}\sum_{y \in \bbZ} \phi((2 y + 1)\,\ee_i) |2 y + 1|^r \P\!\left( |X_{0 i}| > \tfrac12 |2 y + 1| \right)} \nonumber \\ 
&\leq s(r) \E |X_{0 i}|^r \sum_{y \in \bbZ} \phi((2 y + 1)\,\ee_i) \label{high3}
\end{align}
for $i \in \{1, \ldots, d\}$.  Therefore, by \eqref{high1}--\eqref{high3},
\begin{align}
\E \|\XX_1 - \XX_0\|^r
&= \sum_{i = 1}^d \sum_{y \in \bbZ} \phi((2 y + 1) \ee_i) |2 y + 1|^r \P\!\left( |X_{0 i}| > \tfrac12 |2 y + 1| \right) \nonumber \\
&\leq s(r) \sum_{i = 1}^d \E |X_{0 i}|^r \sum_{y \in \bbZ} \phi((2 y + 1) \ee_i) \nonumber \\
&\leq s(r) \E \|\XX_0\|^r \sum_{i = 1}^d \sum_{y \in \bbZ} \phi((2 y + 1) \ee_i) \nonumber \\
&= s(r) \E \|\XX_0\|^r \sum_{\zz \in \bbZ^d} \phi(\zz)
= s(r) \E \|\XX_0\|^r, 
\end{align} 
as desired.
\end{proof}

\subsection{Computation of $s(r)$ for $r \in [1.043, \infty)$}
\label{A:argmax}

Naive application of \refT{T:Isymmunimodalr} to compute $s(r)$ for a given $r \in
[1.043, \infty)$ requires a search through positive integers $k \leq
k_0(r)$, where $k_0(r)$ is given by~\eqref{k0r}.  Our next theorem, the main
result of this subsection, 
shows how the search can be narrowed
considerably and that we can with very few exceptions give simple exact
formulas [namely, \eqref{Kr} and~\eqref{tKr}] for the maximizing value
of~$k$. 

\begin{theorem}
\label{T:Isymmunimodalr+}
Assume $r \in [1.043, \infty)$.
Let $k_0(r)$, $g(k; r)$, and $s(r)$ be as defined in \refT{T:Isymmunimodalr} [see \eqref{sr simpler}--\eqref{k0r}].
Further, let 
\begin{align}
K(r) 
&:= \left\lceil \frac12 \times \frac{3^{r / (r - 1)} - 1}{3^{1 / (r - 1)} - 1} \right\rceil \label{Kr} \\
&= \left\lceil \frac{r - 1}{\ln 3} + \frac{\ln 3}{12} r^{-1} + O(r^{-2}) \right\rceil + 1\mbox{\ as $r \to \infty$} \label{Kr asy} 
\end{align} 
and
\begin{equation}
\label{tKr}
\tK(r) := \left\lceil \frac{r - 1}{\ln 3} \right\rceil + 1.
\end{equation}

\noindent
{\rm (a)}~The sequence $g(k; r)_{k \in \bbN}$ is unimodal with a mode in $[K(r) - 1, \infty)$ and strictly increasing for $k \in [1, K(r) - 1]$.
\smallskip

\noindent
{\rm (b)}~We have
\begin{equation}
s(r) = \max_{k \in \bbN:\,K(r) - 1 \leq k \leq k_0(r)} g(k; r).
\end{equation}

\noindent
{\rm (c)}~For all sufficiently large~$r$ we have
\begin{equation}
s(r) = \max\{g(K(r) - 1; r), g(K(r); r), g(K(r) + 1; r)\}.
\end{equation}

\noindent
{\rm (d)}~For every $j \in \bbZ \cap [11, 1000]$ and $r = j / 10$ we have
\begin{equation}
s(r) = g(K(r); r).
\end{equation} 

\noindent
{\rm (e)}~For every $r \in \bbZ \cap [2, 30000]$ we have
\begin{equation}
s(r) = g(K(r); r).
\end{equation} 

\noindent
{\rm (f)}~For every $j \in \bbZ \cap [11, 1000]$ except $j = 801$ and $r = j / 10$ we have
\begin{equation}
s(r) = g(\tK(r); r),
\end{equation} 
and for $r = 801/10$ we have
\begin{equation}
s(r) = g(\tK(r) - 1; r) = g(72; r).
\end{equation} 

\noindent
{\rm (g)}~For every $r \in \bbZ \cap [2, 30000]$ except $r = 24622$ we have
\begin{equation}
s(r) = g(\tK(r); r),
\end{equation} 
and for $r = 24622$ we have
\begin{equation}
s(r) = g(\tK(r) - 1; r) = g(22411; r).
\end{equation} 
\end{theorem}

To establish \refT{T:Isymmunimodalr+}, we need only prove parts~(a)
and~(c). Then part~(b) follows from \refT{T:Isymmunimodalr} and part~(a),
parts (d)--(e) follow by verification (using exact  arithmetic in
{\tt Mathematica}) that
\begin{equation}
g(K(r); r) > \max\{g(K(r) - 1; r), g(K(r) + 1; r)\}
\end{equation}
for the values of~$r$ in parts (d)--(e), and parts (f)--(g) follow by
verification that $K(r) = \tK(r)$ for the non-exceptional values of~$r$ and
$K(r) = \tK(r) - 1$ for the exceptional values. Parts~(a) and~(c) are
proved later in this 
subsection.

We
have made the proof of parts (d)--(e) [and thus also the proofs of parts (f)--(g)] rigorously complete by using 
{\tt Mathematica} to verify also that for each of the values~$r$ in question the value of $f_1(k_0(r) + 1; r)$ is no greater than the claimed value of $s(r)$.  Parts (d)--(e) then follow rigorously by the same argument as for claim~(d) in the proof of \refT{T:Isymmunimodalr}, namely, that $s(r) \geq g(k; r)$ for any $k \in \bbN$ but also
\begin{equation}
\sup_{k \in \bbN:\,k \geq k_0(r) + 1}\,\max_{y \in \bbZ:\,0 \leq y \leq k - 1} f(k, y; r)
\leq \sup_{k \in \bbN:\,k \geq k_0(r) + 1} f_1(k; r) = f_1(k_0(r) + 1; r). 
\end{equation}
This argument suffices for all values in question except $r = 11/10$ in part~(d), where, as in the argument for claim~(e) in the proof of \refT{T:Isymmunimodalr}, we instead need only (i) show that $f_1(6; 11/10) < g(2; 11/10)$, and indeed
\begin{equation}
f_1(6; 11/10) < 1.058 < 1.065 < g(2; 11/10);  
\end{equation}
and (ii)~verify that $f(k, y; 11/10) < g(2; 11/10)$ for
\begin{equation} 
(k, y) \in \{(2, 0), (3, 0), (3, 1), (4, 0), (4, 1), (4, 2), (5, 0), (5, 1), (5, 2), (5, 3)\}
\end{equation}
(which is simple to do).

\begin{remark}
Despite the impression given by \refT{T:Isymmunimodalr+}(d)--(e), we do not \emph{always} have
$s(r) = g(K(r); r)$.  For example, if $r = 2.04$, then $K(r) = 3$, but $s(r) = g(2; r) > 1.839 > 1.837 > g(3; r)$.
\end{remark}

To prove part~(a) we require some background facts about unimodal sequences established in the next two lemmas.  Recall that a sequence $(a_k)_{k \geq 1}$ is said to be \emph{unimodal} with (positive integer) mode $M$ if 
$(\gD a)_k := a_{k + 1} - a_k \geq 0$ for $1 \leq k \leq M - 1$ and $(\gD a)_k \leq 0$ for $k \geq M$.  Despite the term ``unimodal'', and because of ties, the mode of a unimodal sequence need not be unique; as an extreme example, for a constant sequence every index is a mode.

\begin{lemma}
\label{L:unimodal 1}
Given a sequence $(a_k)_{k \geq 1}$ and a nonnegative integer~$M$, if $(\gD a)_k \geq 0$ for all 
$1 \leq k \leq M - 1$ and $(\gD a)_M \leq 0$ and the sequence $(\gD a)_k$ is nonincreasing for $k \in [M, \infty)$, then $(a_k)$ is unimodal with mode~$M$.
\end{lemma}

\begin{proof}
This is clear from the definition of unimodal sequence, because for $k \geq M$ we have $(\gD a)_k \leq (\gD a)_M \leq 0$.
\end{proof}

\begin{lemma}
\label{L:unimodal 2}
Given a sequence 
$(a_k)_{k \geq 1}$, an integer $M \geq 2$, and an integer $\tM \in [1, M - 1]$, suppose that the sequence 
$((\gD a)_k)_{k \geq 0}$ is unimodal with mode~$\tM$ and satisfies $(\gD a)_1 \geq 0$, $(\gD a)_{M - 1} \geq 0$, and $(\gD a)_M \leq 0$.  Then $(a_k)$ is unimodal with mode~$M$.  
\end{lemma}

\begin{proof}
This follows by applying \refL{L:unimodal 1}.
\end{proof}

We are ready now for the proof of \refT{T:Isymmunimodalr+}(a).

\begin{proof}[Proof of \refT{T:Isymmunimodalr+}(a)]
Fix $r \in [1.043, \infty)$ and drop it from the notation (when not needed for clarity) for this proof.  
[Thus, for example, we write $S(k)$ rather than $S(k; r)$.]  
A brief calculation, notably using $S(k + 1) = S(k) + (k + 1)^r$, yields
\begin{align}\label{gdg}
(\gD g)_k 
&= g(k + 1) - g(k) \nonumber \\ 
&= \frac{(2 k + 1)^r - (2 k - 1)^r}{S(k) S(k + 1)} 
\left\{ S(k) - (k + 1)^r \left[ \left( \frac{2 k + 1}{2 k - 1} \right)^r - 1 \right]^{-1} \right\} \nonumber \\
&= \frac{(2 k + 1)^r - (2 k - 1)^r}{S(k) S(k + 1)} (\gD \tg)_k 
\end{align}
for $k \geq 1$, where for $k \geq 1$ we define
\begin{equation}
\tg_k := \sum_{j = 1}^{k - 1} \left\{ S(j) - (j + 1)^r \left[ \left( \frac{2 j + 1}{2 j - 1} \right)^r - 1 \right]^{-1} \right\}.   
\end{equation}
Therefore~$g$ is unimodal with mode~$M$ if and only if~$\tg$ is unimodal with mode~$M$.

Thus we can prove unimodality for~$g$ by proving unimodality for~$\tg$, and we carry out the latter by applying \refL{L:unimodal 2} to the sequence $a = \tg$.  There are four hypotheses to verify in \refL{L:unimodal 2}.  The first is that the sequence $h = (h_k)_{k \geq 1}$ defined by
\begin{equation}
\label{hk}
h_k := (\gD \tg)_k = S(k) - (k + 1)^r \left[ \left( \frac{2 k + 1}{2 k - 1} \right)^r - 1 \right]^{-1}
\end{equation}  
is unimodal.  For this, observe that for $k \geq 1$ we have
\begin{align}\label{hk2}
\lefteqn{\left[ \left( \frac{2 k + 1}{2 k - 1} \right)^r - 1 \right] 
\left[ \left( \frac{2 k + 3}{2 k + 1} \right)^r - 1 \right] (\gD h)_k} \nonumber \\  
&= (k + 1)^r \left( \frac{2 k + 1}{2 k - 1} \right)^r \left[ \left( \frac{2 k + 3}{2 k + 1} \right)^r - 1 \right] - (k + 2)^r \left[ \left( \frac{2 k + 1}{2 k - 1} \right)^r - 1 \right] \nonumber \\
&= (k + 1)^r \left[ \left( \frac{2 k + 3}{2 k - 1} \right)^r - \left( \frac{2 k + 1}{2 k - 1} \right)^r \right] - (k + 2)^r \left[ \left( \frac{2 k + 1}{2 k - 1} \right)^r - 1 \right] \nonumber \\
&= (2 k - 1)^{-r} \times \nonumber \\ 
&\quad \Big\{ (k + 1)^r \left[ (2 k + 3)^r - (2 k + 1)^r \right] 
- (k + 2)^r \left[ (2 k + 1)^r - (2 k - 1)^r \right] \Big\} \nonumber \\
&= 2^{-r} (2 k - 1)^{-r} (2 k + 2)^r (2 k + 4)^r (\gD \tth)_k
\end{align}
where
\begin{align}
\tth_k
&:= (2 k + 2)^{-r} \left[ (2 k + 1)^r - (2 k - 1)^r \right] \nonumber \\ 
&= \left( 1 - \frac{1}{2 k + 2} \right)^r - \left(1 - \frac{3}{2 k + 2} \right)^r.
\end{align}
Thus~$h$ is unimodal with mode~$M$ if and only if $\tth$ is unimodal with mode~$M$.

The function $t \mapsto (1 - t)^r - (1 - 3 t)^r$ is strictly increasing for $t \in [0, t_r]$ and strictly decreasing for 
$t \in [t_r, 1/3]$, where $t_r := \frac{3^{1 / (r - 1)} - 1}{3^{r / (r - 1)} - 1}$.  It follows that $(\gD \tth)_k > 0$ if 
$k \leq \left\lfloor \frac{1}{2 t_r} \right\rfloor - 2$ and $(\gD \tth)_k < 0$ if $k \geq \left\lceil \frac{1}{2 t_r} \right\rceil - 1 = K(r) - 1$.  If $1 / (2 t_r) \in \bbN$, then we find that $\tth$ is unimodal with unique mode $K(r) - 1$.  If $1 / (2 t_r)$ is not an integer, then we find that $\tth$ is unimodal, with (i)~unique mode $K(r) - 2$ if 
$\tth_{K(r) - 2} > \tth_{K(r) - 1}$, (ii)~unique mode $K(r) - 1$ if $\tth_{K(r) - 2} < \tth_{K(r) - 1}$, and (iii)~precisely these two values as the modes if $\tth_{K(r) - 2} = \tth_{K(r) - 1}$.
By \eqref{hk2}, the same holds for the sequence $h$.

In order to use \refL{L:unimodal 2}, we also show that $h_1 \geq 0$ and $\liminf_{k\to\infty} h_k < 0$.  Indeed, from~\eqref{hk} we have
\begin{equation}\label{h1>}
h_1 = 1 - \frac{2^r}{3^r - 1} > 0
\end{equation}
and for large~$k$ the asymptotics
\begin{equation}
h_k = (1 + o(1)) \frac{k^{r + 1}}{r + 1} - (1 + o(1)) k^r \frac{k}{r} \sim - \frac{k^r}{r (r + 1)}.
\end{equation} 

It follows that $(\gD\tg)_k=h_k\ge h_1>0$ for $1\le k\le K(r)-2$, and 
\refL{L:unimodal 2} applies to $(\tg_k)_{k\ge1}$ with some $M\ge K(r)-1$.
The result follows by \eqref{gdg}.
\end{proof}

We can now use \refT{T:Isymmunimodalr+}(a) to prove \refT{T:Isymmunimodalr+}(c).

\begin{proof}[Proof of \refT{T:Isymmunimodalr+}(c)]
By \refT{T:Isymmunimodalr+}(a), it suffices to prove that
\begin{equation}
g(K(r) + 2; r) < g(K(r) + 1; r), 
\end{equation}
or, equivalently, writing~$K$ as abbreviation for $K(r)$, that
\begin{equation}
\label{less than 1}
\rho(r) := (K + 2)^{-r} \left[ \left( \frac{2 K + 3}{2 K + 1} \right)^r - 1 \right] S(K + 1; r) < 1
\end{equation}
for all sufficiently large~$r$.

This requires \emph{three}-term 
asymptotics for each of the three factors in~\eqref{less than 1}.  We will use notation as in \refL{L:helpful4}, with slight modification.  Since the argument~$k$ of $S(k; r)$ appearing in~\eqref{less than 1} is $K + 1$, we will define $\gam \equiv \gam(r)$ by $\gam := r / (K + 1)$ and $\xi \equiv \xi(r)$ by $\xi := e^{-\gam}$.  

The first two factors can be handled in rather straightforward fashion, and the result is  that
\begin{equation}
(K + 2)^{-r} = (K + 1)^{-r} [\xi + \tfrac12 \gam^2 \xi r^{-1} + (1 + o(1)) \tfrac{1}{24} \gam^3 (3 \gam - 8) \xi r^{-2}]
\end{equation}
and 
\begin{equation}
\left( \frac{2 K + 3}{2 K + 1} \right)^r - 1
= \xi^{-1} (1 - \xi) + (1 + o(1)) \tfrac{1}{12} \gam^3 \xi^{-1} r^{-2}.  
\end{equation}

The third factor can be handled by elaborating on the proof of \refL{L:helpful4}, resulting  in
\begin{equation} 
S(K + 1; r) 
= (K + 1)^r [\Sigma_0 - \tfrac12 \gam^2 \Sigma_2 r^{-1}  
+ (\tfrac18 \gam^4 \Sigma_4 - \tfrac13 \gam^3 \Sigma_3) r^{-2} + o(r^{-2})]
\end{equation}
where $\Sigma_m := \sum_{j = 0}^{\infty} e^{- \gam j} j^m$ and in particular
\begin{align}
\Sigma_0 &= (1 - \xi)^{-1}, \nonumber \\
\Sigma_2 &= (1 - \xi)^{-3} \xi (1 + \xi), \nonumber \\
\Sigma_3 &= (1 - \xi)^{-4} \xi (1 + 4 \xi + \xi^2), \nonumber \\
\Sigma_4 &= (1 - \xi)^{-5} \xi (1 + \xi) (1 + 10 \xi + \xi^2).
\end{align}

Upon multiplying the three factors, the upshot is  that
\begin{equation}
\rho(r) = 1 + \tfrac12 \gam^2 (1 - 3 \xi) (1 - \xi)^{-2} r^{-1} + a_2 r^{-2} + o(r^{-2}),
\end{equation}
where
\begin{align}
a_2 
&= \tfrac{1}{24} \gam^3 (3 \gam - 8) + \tfrac{1}{12} \gam^3 (1 - \xi)^{-1} - \tfrac14 \gam^4 \xi (1 + \xi) (1 - \xi)^{-2} 
\nonumber \\ 
&{} \quad {} - \tfrac13 \gam^3 \xi (1 + 4 \xi + \xi^2) (1 - \xi)^{-3} 
+ \tfrac18 \gam^4 \xi (1 + \xi) (1 + 10 \xi + \xi^2) (1 - \xi)^{-4}.
\end{align}
If, recalling~\eqref{Kr asy}, we now substitute $K + 1 = \frac{r - 1}{\ln 3} + 2 + \theta(r) + O(r^{-1})$, where $0 \leq \theta(r) < 1$, we obtain
\begin{equation}
\rho(r) = 1 - \tfrac98 (\ln 3)^4 [1 + \theta(r)] r^{-2} + o(r^{-2}),
\end{equation}
which is smaller than~$1$ for all sufficiently large~$r$. 
\end{proof}

\subsection{Computation of $s(r)$ for $r \in (1, 1.043)$}
\label{A:compute}

We have discussed computation of $s(r)$ for $r \geq 1.043$ in \refT{T:Isymmunimodalr} and \refS{A:argmax}.
The next result shows (recall also \refL{L:helpful1}) that
we can rigorously compute $s(r)$ exactly for any given $r \in (1, 1.043)$, by
considering the maximum of $g(k; r)$ for the first approximately $7.325 (r - 1)^{-1}$
values of~$k$. 

\begin{proposition}
\label{P:compute}
Recall the definition~\eqref{fkyr} of $f(k, y; r)$.  For $r \in (1, 1.043)$ we have
\begin{align}
s(r) 
&= \max_{k \in \bbN:\,k \leq k_1(r)}\,\max_{y \in \bbZ:\,0 \leq y \leq k - 1} f(k, y; r) \label{10001043} \\ 
&= \max_{k \in \bbN:\,k \leq k_1(r)}\,
\max\left\{ f\!\left( k, \left\lfloor \frac{2 r k - 1}{2 (r + 1)} \right\rfloor; r \right), 
f\!\left( k, \left\lceil \frac{2 r k - 1}{2 (r + 1)} \right\rceil; r \right) \right\}, 
\label{10001043 better}
\end{align}
where
\begin{align}
\label{k1r}
k_1(r) 
&:= \left\lceil \frac12 \left\{ \left[ (2^r + 1)^{-1}  \left( \frac{3 (r + 1)}{2 r} \right)^r \right]^{1 / (r + 1)} - 1 \right\}^{-1} \right\rceil - 1 \\ 
& \sim \frac{1}{\ln 3  -\frac12 - \frac23 \ln 2} (r - 1)^{-1}\mbox{\rm \ as $r \downarrow 1$}.
\end{align}
\end{proposition} 

\begin{proof}
Observe that the expression appearing in curly braces in~\eqref{k1r} is indeed strictly positive  for 
$r \in (1, 2.14]$.  Also, as noted in the proof of \refL{L:helpful1}, the expression $f(k, y; r)$ given in~\eqref{fkyr} is log-concave in real~$y$ and maximized over real~$y$ at $y = \frac{2 r k - 1}{2 (r + 1)}$ 
[recall~\eqref{real maximizing y}]; the maximum over integer~$y$ is thus at either the floor or ceiling of this value.  This leads us from~\eqref{10001043} to~\eqref{10001043 better}. 

The strategy for proving~\eqref{10001043} is to note that 
\begin{equation}
\label{sr LB}
s(r) \geq g(2;r) = \frac{3^r}{S(2; r)}
\end{equation}
for all $r \in (1, \infty)$
while also upper-bounding 
\begin{equation}
s_1(r) := \sup_{k > k_1(r)}\,\max_{0 \leq y \leq k - 1} f(k, y; r)
\end{equation}
as follows, using~\eqref{fkyr}--\eqref{f1kr}, \refL{L:helpful2}(b), and the simple integral-comparison lower  bound 
$S(k; -r) \geq (r + 1)^{-1} k^{r + 1}$:
\begin{equation}
\label{s1r bound}
s_1(r) \leq f_1(k_1(r) + 1;r) \leq \left( \frac{2 r}{r + 1} \right)^r \left[1 + \frac{1}{2 (k_1(r) + 1)} \right]^{r + 1}.
\end{equation}
With the definition~\eqref{k1r} of $k_1(r)$, this implies
\begin{equation}
\label{s1r bound b}
s_1(r) \leq \frac{3^r}{1 + 2^r} = g(2; r) \leq s(r),
\end{equation}
completing the proof.
\end{proof}

\begin{remark}
(a) We
have used~\eqref{10001043 better} and {\tt Mathematica} (i)~with exact arithmetic to verify (rigorously) that 
$s(r) = g(2; r) = 3^r / (1 + 2^r)$ for $r = 1.001, 1.002, \ldots, 1.043$ and (ii)~with floating-point arithmetic to verify that $s(r) = g(2; r)$ for $r = 1 + 2^{-j}$ for $j = 10, \ldots, 25$.  This evidence is the basis of our strong belief that 
\refConj{Conj:s g} is correct.

(b)~\refP{P:compute} is much ado about nothing if \refConj{Conj:s g} is correct.  Indeed, in that case we would  find
\begin{equation}
s(r) = g(2, r) = \frac{3^r}{1 + 2^r}\mbox{\ for $r \in [1, 2]$}.
\end{equation}
In particular (compare \refR{R:bounds comparison}), as $r \downarrow 1$ we would then have
\begin{equation}
s(r) = 1 + (1 + o(1)) (\ln 3 - \tfrac23 \ln 2) (r - 1) \approx 1 + 0.637 (r - 1),
\end{equation}
whereas the corresponding supremum in the absolutely continuous symmetric
unimodal case equals (see  \refT{T:unimodal} and \refR{R:sharp})
\begin{equation}
\left( \frac{2 r}{r + 1} \right)^r  = 1 + (1 + o(1)) \frac12 (r - 1)
\end{equation} 
and in the general case equals (see  \refT{T:main} and \refR{R=})
\begin{equation}
2^{r - 1} = 1 + (1 + o(1) (\ln 2) (r - 1) \approx 1 + 0.693 (r - 1).
\end{equation} 

(c) We
could narrow the gap $(1, 1.043)$ in \refT{T:Isymmunimodalr}, but we do not know how to eliminate the gap entirely.
The difficulty is that we do not know a better tractable bound on $f(k, y; r)$ than $f_1(k, r)$ [recalling~\eqref{fkyr}], and for each $k \in \bbN$ the inequality $f_1(k; r) \leq g(2; r)$ fails for all $r > 1$ sufficiently close to~$1$, since 
$f_1(k; \cdot)$ and $g(2; \cdot)$ are both continuous and
\begin{equation}
f_1(k; 1) = \frac{4 k^2 + 4 k + 1}{4 k^2 + 4 k} > 1 = g(2; 1).
\end{equation}
\end{remark}

\newcommand\vol{\textbf}

\end{document}